\documentclass[10pt,reqno,final]{amsart}
\usepackage{ifplatform}
\usepackage{epsfig,amssymb,amsmath,version}
\usepackage{amssymb,version,graphicx,fancybox,mathrsfs,multirow}
\usepackage{url,hyperref}
\usepackage[notcite,notref]{showkeys}

\usepackage{subfigure}%,mathabx
\usepackage{color,xcolor}
\usepackage{cases}
\usepackage{mathtools}
\usepackage{wrapfig}

\catcode`\@=11 \theoremstyle{plain}
\@addtoreset{equation}{section}   % Makes \section reset 'equation' counter.
\renewcommand{\theequation}{\arabic{section}.\arabic{equation}}
\@addtoreset{figure}{section}
\renewcommand\thefigure{\thesection.\@arabic\c@figure}
\renewcommand{\thefigure}{\arabic{section}.\arabic{figure}}
\newtheorem{thm}{\bf Theorem}

\newenvironment{theorem}{\begin{thm}} {\end{thm}}
\newtheorem{cor}{\bf Corollary}
\newtheorem{prop}{Proposition}[section]

\newtheorem{lmm}{\bf Lemma}
\renewcommand{\thelmm}{\arabic{section}.\arabic{lmm}}
\newenvironment{lemma}{\begin{lmm}}{\end{lmm}}
\theoremstyle{remark}
\newtheorem{rem}{\bf Remark}[section]
\theoremstyle{definition}
\newtheorem{defn}{\bf Definition}[section]

 \numberwithin{table}{section}

\def \rd {{\rm d}}
\renewcommand \wedge \times

\begin{document}
\bibliographystyle{plain}
%\bibliographystyle{unsrt}

%\graphicspath{{../Figures/}}
\baselineskip 13pt

\title[Approximation by Legendre expansions] {Optimal  error estimates for
 Legendre  expansions  of singular functions with fractional derivatives of bounded variation}
\author[
	W. Liu, \,   L. Wang\,  $\&$\, B. Wu
	]{\;\; Wenjie Liu$^\dagger$,   \;\;  Li-Lian Wang$^\ddagger$ \;\; and\;\; Boying Wu$^\dagger$
		}

	\thanks{${}^\dagger$Department of Mathematics, Harbin Institute of Technology, 150001, China.  The research of the first author was supported by the China Postdoctoral Science Foundation Funded Project (No. 2017M620113), the National Natural Science Foundation of China (Nos. 11801120, 71773024 and 11271100), the Fundamental Research Funds for the Central Universities (Grant No.HIT.NSRIF.2020081) and the Natural Science Foundation of Heilongjiang Province of China (Nos. A2016003 and G2018006). Emails: liuwenjie@hit.edu.cn (Wenjie Liu) and mathwby@hit.edu.cn (Boying Wu).\\
		\indent ${}^\ddagger${\em Corresponding author}.  Division of Mathematical Sciences, School of Physical
		and Mathematical Sciences, Nanyang Technological University,
		637371, Singapore. The research of the second author is partially supported by Singapore MOE AcRF Tier 2 Grants: MOE2018-T2-1-059 and MOE2017-T2-2-144. Email: lilian@ntu.edu.sg.
%		\indent ${}^{3}$State Key Laboratory of Computer Science/Laboratory of Parallel Computing,  Institute of Software, Chinese Academy of Sciences, Beijing 100190, China. The research of this author is partially   supported by the National Natural Science Foundation of China (91130014, 11471312 and 91430216).
	}
	
\begin{abstract} We present a new fractional Taylor formula for
singular functions whose  Caputo fractional derivatives are of bounded variation.  It bridges and ``interpolates" the usual Taylor formulas with two consecutive integer orders.  This enables us to obtain an analogous formula for the Legendre expansion coefficient of this type of singular functions, and further derive the optimal (weighted) $L^\infty$-estimates and $L^2$-estimates of the Legendre polynomial approximations.  This set of results can enrich the existing theory for $p$ and $hp$ methods for singular problems, and
 answer some open questions posed in some recent literature.
\end{abstract}
\maketitle

%\vspace*{-10pt}

\section{Introduction}

The study of  Legendre approximation to singular functions  is a subject of fundamental importance in the theory and applications of $hp$ finite element methods. We refer to   the seminal series  of papers by
Gui and Babu\v{s}ka \cite{Gui1986NM-Part1,Gui1986NM-Part2,Gui1986NM-Part3}  and many other developments in e.g.,  \cite{Schwab1998Book,Babuska2000NM,Babuska2001SINUM}.
In particular, the very recent work of  Babu\v{s}ka and Hakula \cite{Babuska19CMAME-Pointwise}  provided a review of known/unknown results and posed
 a few open questions  on the pointwise error estimates of Legendre expansion of a typical  singular function discussed in
 \cite{Gui1986NM-Part1}:
 \begin{equation}\label{uxy1}
 u(x)=(x-\theta)^\mu_+=\begin{cases}
 0,  & -1< x\le \theta,\\[2pt]
 (x-\theta)^\mu,    & \theta< x<1,
 \end{cases}
 \quad  |\theta|\le 1,\;\;\; \mu>-1.
 \end{equation}
One significant development along this line is the $hp$ approximation theory in the framework of  Jacobi-weighted Besov spaces \cite{Babuska2000NM,Babuska2001SINUM,Babuska2002MMMAS,Guo2013JAT-Direct}.
Such Besov spaces are defined through space interpolation of Jacobi-weighted Sobolev spaces with integer regularity indices using the $K$-method.  It is important to point out  that the non-uniformly Jacobi-weighted Sobolev spaces has been employed in spectral approximation theory    \cite{Funaro1992Book,Schwab1998Book,Guo2004JAT,Guo06JSC-Optimal,Shen2011Book}.

\subsection{Related works} Different from  the Sobolev-Besov framework,  Trefethen  \cite{Trefethen2008SIREV,Trefethen2013Book} characterised
the regularity of singular functions by using the space of absolute continuity  and bounded variation (AC-BV),
in the study of Chebyshev expansions of such functions.
%is the AC$_m$-BV-regularity (cf. \!), i.e.,  $u,u',\cdots, u^{(m-1)}\in {\rm AC}(\bar\Omega)$ and $u^{(m)}\in {\rm BV}(\bar \Omega)$ for integer $m\ge 0$ and $\Omega:=(-1,1),$ where
%${\rm AC}(\bar\Omega)$ (resp.\!\! ${\rm BV}(\bar\Omega)$) the space of absolutely continuous functions (resp. functions of  bounded variation).
One motivative example therein is $u(x)=|x|$ in $\Omega=(-1,1)$ which has the regularity:   $u,u'\in {\rm AC}(\bar \Omega)$ and
$u''\in {\rm BV}(\bar \Omega)$ (where the integration of the norm  is in the Riemann-Stieltjes (RS) sense).
As a result,  the maximum error of its Chebyshev expansion can attain  optimal order  (but can only be suboptimal in a usual Sobolev framework). There have been many follow-up  works on the improved error estimates of  Chebyshev approximation or more general Jacobi polynomial approximation  under this AC-BV framework (see, e.g., \cite{Majidian2017ANM,Wang18JCAM-convergence,Wang18AML-new,Xiang20NM-Optimal}).  However, the regularity index and the involved derivatives are of integer order, so it is not suitable to best  characterise the regularity of many singular functions, e.g., \eqref{uxy1} and $u(x)=|x|^\mu$ with non-integer  $\mu.$ In other words,  if one naively applies the estimates,  then the loss of order might occur.
Nevertheless,  the solutions of many singular problems (in irregular domains or with singular coefficients/operators among others)
typically exhibit this kind of singularities.

 To fill this gap, we introduced  for the first time in \cite{Liu19MC-Optimal} certain fractional Sobolev-type spaces
and derived optimal Chebyshev polynomial approximation to  functions with interior and endpoint singularities within this new framework.  This study also inspired the discovery of generalised Gegenbauer functions of fractional degree, as an analysis tool  and a class of special functions with rich properties  \cite{Liu20JAT-Asymptotics}.

\subsection{Our contributions} Undoubtedly,  the Taylor formula  plays a foundational role in numerical  analysis and algorithm development.  We present  a new fractional Taylor formula for  AC-BV functions with fractional regularity index (see Theorem \ref{NewFTFV1}) that  ``interpolates"  and seamlessly  bridges the Taylor formulas of two consecutive integer orders.  From this tool, we can derive  an analogous formula for the Legendre expansion coefficient of the same class of functions, which
turns out  the cornerstone of all the analysis.
Then we obtain a set of optimal Legendre approximation results in $L^\infty$- and $L^2$-norms for functions with both interior and endpoint singularities.
We highlight that the use of  function space involving fractional integrals/derivatives to characterise regularity follows  that in  \cite{Liu19MC-Optimal}, but we further refine this framework by introducing the Caputo derivative. When the fractional regularity index takes  integer value, it reduces to the AC-BV space in Trefethen  \cite{Trefethen2008SIREV,Trefethen2013Book} (with adaption to the  Legendre approximation).  We  point out that  the argument for the Legendre approximation herein  is different from that for the Chebyshev approximation in   \cite{Liu19MC-Optimal}.   It is also noteworthy that  Babu\v{s}ka and Hakula \cite{Babuska19CMAME-Pointwise} discussed the point-wise error estimates
of  the Legendre expansion for the specific  function \eqref{uxy1} (which is also the subject of \cite{Gui1986NM-Part1}) including known and unknown results. In fact, it appears necessary to study the point-wise error in the Legendre or other Jacobi cases. For example, the estimating the $L^\infty$-error like the Chebyshev expansion can only lead to suboptimal results for
functions with the interior singularity, e.g., $u(x)=|x|,$ as  a loss of half order occurs.  It was observed numerically, but how to obtain optimal estimate appears open (see, e.g., \cite{Wang18AML-new}). Here, we shall provide an answer to this, and also to some conjectures in  \cite{Babuska19CMAME-Pointwise}.
We remark that we aim at deriving sharp and optimal estimates valid for all polynomial orders.  According to \cite{Babuska19CMAME-Pointwise},  in most applications the polynomial orders are relatively small compared to those in the asymptotic range, while  the existing theory does not address the behaviour of the pre-asymptotic error.  As a result, our arguments and results are  different from those in  \cite{Xiang20NM-Optimal},  where  some  asymptotic formulas were employed to derive Jacobi approximation of specific singular functions for large polynomial orders.
As a final remark,   this paper will be largely devoted to the $L^\infty$- and $L^2$-estimates of the  finite Legendre expansions, which lay the groundwork for establishing the approximation theory of other orthogonal projections, interpolations and quadrature for singular functions.   Indeed, these results can  enrich the theoretical foundation of   $p$ and $hp$ methods (cf. \cite{Funaro1992Book,Schwab1998Book,Bernardi1997Book,Canuto2006Book,Hesthaven2007Book,Shen2011Book}).
%We leave the applications to singular problems in a future work.
In a nutshell, the present study together with \cite{Liu20JAT-Asymptotics,Liu19MC-Optimal}  is far from being the last word on this subject.

The rest of the paper is organised as follows. In section \ref{sect:fintderG}, we derive the fractional Taylor formula for the AC-BV functions and present some  preliminaries to pave the way for all forthcoming  discussions. In section \ref{sect3main}, we obtain the main results on Legendre approximation of functions with interior singularities and extend the tools to study the
endpoint singularities in  section \ref{sect4main}.

\section{Fractional integral/derivative formulas of GGF-Fs}\label{sect:fintderG}
\setcounter{equation}{0}
\setcounter{lmm}{0}
\setcounter{thm}{0}

%In this section, we show that  GGF-Fs enjoy some remarkable  fractional  calculus properties,
%which are  important  for the error analysis. % and are useful for spectral algorithms.
%Before pre we recall the definitions of RL fractional integrals/derivatives.
 %, and introduce the related spaces of functions.
In this section, we make necessary preparations for the forthcoming   discussions. More precisely, we first introduce several spaces of functions that will be used to characterise the regularity of the class of functions of interest.  We then recall the definition of  the Riemann-Liouville (RL) fractional integrals, and  present   a useful RL fractional integration parts formula.  Finally, we collect some relevant properties of generalised Gegenbauer functions of fractional degree (GGF-Fs), which were first introduced and studied in
\cite{Liu19MC-Optimal,Liu20JAT-Asymptotics}.

\subsection{Spaces of functions}\label{Subsection2.1}
%In this subsection, we make necessary preparations, and start with introducing related spaces of functions.
 Let $\Omega=(a,b)\subset {\mathbb R}$ be a finite open interval. For  real $p\in [1, \infty],$ let $L^p(\Omega)$ (resp. $W^{m,p}(\Omega)$ with $m\in {\mathbb N},$ the set of all positive integers)
 be the usual
$p$-Lebesgue space (resp. Sobolev space), equipped  with the norm $\|\cdot\|_{L^p(\Omega)}$ (resp. $\|\cdot\|_{W^{m,p}(\Omega)}$), as  in Adams \cite{Adams1975Book}.

 Let $C(\bar \Omega)$ be the classical space of continuous functions, and  ${\rm AC}(\bar \Omega)$
the space of  absolutely continuous functions on $\bar \Omega.$  It is known that
every absolutely continuous function is uniformly continuous (but the converse is not true), and hence continuous (cf. \!\cite[p.\! 483]{Ponnusamy2012Book}). It is known  that {\em a real function  $f(x)\in {\rm AC}(\bar \Omega)$ if and only if $f(x)\in L^1(\Omega)$, $f(x)$ has a derivative $f'(x)$
almost everywhere on $[a,b]$ such that   $f'(x)\in L^1(\Omega), $ and $f(x)$ has the integral representation:}
\begin{equation}\label{AVint}
f(x)=f(a)+\int_a^x f'(t)\,\rd t,\quad \forall\,  x\in [a,b],
\end{equation}
(cf.\! \cite[Chap.\! 1]{Samko1993Book} and  \cite[p.\! 285]{Lang1993Book}).

\begin{comment}
According to
\cite[Ch. 3]{Leoni2009Book} (also see \cite[P. 206]{Brezis2011Book} or  \cite[Ch. 1]{Samko1993Book}), a real function $f(x)\in {\rm AC}(\bar \Omega),$ if for any $\varepsilon>0,$ there exists $\delta>0,$ such that for every finite sequence of disjoint intervals $(a_k, b_k)\subset \Omega$ such that $\sum_k |b_k-a_k|<\delta,$ we have $\sum_k |f(b_k)-f(a_k)|<\varepsilon.$

% As remarked by  \cite[Remark 8]{Brezis2011Book} (also see \!\cite[Sec.\! 7.2]{Leoni2009Book}),
%the space   ${\rm AC}(\bar \Omega)$ coincides with  $W^{1,1}(\Omega)$ characterized by
%$$
%W^{1,1}(\Omega)=\Big\{f\in L^1(\Omega): \exists\, g\in L^1(\Omega) \;  \text{such that} \; \int_{\Omega} f\varphi' dx
%=-\int_\Omega g \varphi dx,\; \forall \varphi\in C^1_c(\bar \Omega) \Big\},
%$$
%where $C^1_c(\bar \Omega)\subset C^1(\bar \Omega)$ with all functions being compactly supported in $\bar \Omega.$
Recall that (cf. \cite[Ch. 1]{Samko1993Book} or  \cite[P. 285]{Lang1993Book}):  a real function  $f(x)\in {\rm AC}(\bar \Omega)$ if and only if $f(x)\in L^1(\Omega)$, $f(x)$ has a derivative $f'(x)$
almost everywhere on $[a,b]$ such that   $f'(x)\in L^1(\Omega), $ and $f(x)$ has the integral representation:
\begin{equation*}\label{AVint}
f(x)=f(a)+\int_a^x f'(t)\,dt,\quad \forall x\in [a,b].
\end{equation*}
\end{comment}

Let ${\rm BV}(\bar \Omega)$ be the space of functions of bounded  variation on $[a,b].$  We say that  {\em a real function $f(x)\in {\rm BV}(\bar \Omega), $ if there exists a constant $C>0$ such that
$$V({\mathbb P};f):=\sum\limits_{i=0}^{k-1}|f(x_{i+1})-f(x_i)|\le C$$
 for every finite partition
${\mathbb P}=\{x_0, x_1,\cdots, x_k\}$ {\rm(}satisfying $x_i<x_{i+1}$ for all $0\le i\le k-1${\rm)} of $[a,b].$}
Then  the total variation of $f$ on $[a,b]$ is defined as $V_{\bar \Omega}[f]:=\sup\{V({\mathbb P};f)\},$ where the supreme is taken over all the partitions of $\bar \Omega$ (cf.
\cite[p.\! 207]{Brezis2011Book} or \cite[Chap.\! X]{Lang1993Book}).
% (also see
%\cite[Ch. 11]{Ponnusamy2012Book} and \cite[Ch. X]{Lang1993Book}),.
An important characterisation of a BV-function  is the Jordan decomposition
(cf. \!\cite[Thm.\! 11.19]{Ponnusamy2012Book}):  {\em a function is of bounded variation if and only if it can be expressed as the difference of two increasing functions on $[a,b].$}
As a result,   every function in  ${\rm BV}(\bar\Omega)$ has at most a countable number of discontinuities, which are  either jump or removable discontinuities,  % (see,  e.g., \cite[Thm 2.8]{Wheeden2015Book}).
so it is differentiable almost everywhere.  Indeed, according to
 \cite[p.\! 223]{Appell2014Book},  {\em  if $f \in {\rm AC}(\bar \Omega),$ then}
\begin{equation*}\label{Vnorm}
V_{\bar \Omega}[f]= \int_{\Omega} |f'(x)|\, {\rm d}x.
\end{equation*}
%{\rm (ii)}  if  $f\in {\rm AC}(\bar \Omega),$ we even have equality }
%\begin{equation}\label{VeqnA}
%V_{\bar \Omega}[f] = \int_\Omega |f'(x)|\,{\rm d}x.
%\end{equation}
In fact, we have  ${\rm BV}(\bar \Omega) \subset {\rm AC}(\bar \Omega)={\rm W}^{1,1}( \Omega)$  in the sense that   every $f(x) \in  {\rm AC}(\bar \Omega)$ has an almost everywhere classical derivative $f'\in L^{1}(\Omega)$ (cf.  \eqref{AVint}) and  $f'(x)$ is the weak derivative
of $f(x).$ Conversely, even $f \in W^{1.1}(\Omega),$ modulo a modification on a set of measure
zero, is an absolutely continuous function (cf.\! \cite[p.\! 206]{Brezis2011Book} and   \cite[p.\! 84; p.\! 96]{Buttazzo1998Book}).

%
%{\color{red}
%Choose $k+1$ points of $(a, b)$ (not necessarily uniformly spread)
%\begin{equation*}
%{\mathbb P}:a<x_{0}<x_{1}<\cdots<x_{k-1}<x_{k}<b.
%\end{equation*}
%Denote $\Delta x_{i}=x_{i}-x_{i-1}$ and $\lambda(\mathbb P)=\max\{x_0-a,\Delta x_{0},\ldots,\Delta x_{k}, b-x_k \}$.
%The integral of $f(x)$ over $(a, b)$ is defined by
%\begin{equation*}
%\int_{a}^{b} f(x) \mathrm{d} x:=\lim _{\lambda(\mathbb P) \rightarrow 0 } \sum_{i=1}^{k} f\left(\xi_{i}\right) \Delta x_{i}, \quad \forall
%\xi_{i}\in [x_{i-1},x_i].
%\end{equation*}
%Similarly, the Riemann–Stieltjes integral defined by
%\begin{equation*}
%\int_a^b f \rd g:=\lim _{\lambda(\mathbb P) \rightarrow 0 } \sum_{i=1}^{k} f\left(\xi_{i}\right) (g( x_{i})-g( x_{i-1})), \quad \forall
%\xi_{i}\in [x_{i-1},x_i].
%\end{equation*}
%For $f\ge 0,$
%\begin{equation*}
%\int_a^b f |\rd g|:=\lim _{\lambda(\mathbb P) \rightarrow 0 } \sum_{i=1}^{k} f\left(\xi_{i}\right)| g( x_{i})-g( x_{i-1})|, \quad \forall
%\xi_{i}\in [x_{i-1},x_i].
%\end{equation*}
%Obvious that
%$$\int_a^b f \rd g\le\int_a^b |f| \, |\rd g| ,\quad  \int_a^b  |\rd g|=V_g(\bar \Omega).$$
%If $g\in {\rm AC}(\bar \Omega),$ then $\int_a^b f \rd g=\int_a^b f g' \rd x.$
%}

For BV-functions,  we can define  the Riemann-Stieltjes (RS) integral  (cf.\! \cite[Chap.\!\! X]{Lang1993Book}).  A function $f(x)$ is said to be  RS($g$)-integrable, if $\int_\Omega f \rd g<\infty$ for  $g\in {\rm BV} (\bar \Omega).$
From  \cite[Prop.\! 1.3]{Lang1993Book}, we have  that
{\em  if $f$ is  {\rm RS($g$)-integrable}, then % and $g\in {\rm BV} (\bar \Omega),$ then
\begin{equation}\label{integralProp}
\Big|\int_\Omega f(x)\, \rd g(x)\Big|\le \|f\|_\infty\, V_{\bar \Omega}[f],\quad  \int_\Omega  |\rd g(x)|=V_{\bar \Omega}[g],
\end{equation}
where $\|f\|_\infty$ is the $L^\infty$-norm of $f$ on $[a,b].$}

In the analysis, we shall also use the splitting rule of a RS integral, which is different from the usual integral.
\begin{lmm}[{see Carter and Brunt \cite[Thm 6.1.1 \& Thm 6.1.6]{Carter2000Book}}]\label{LmmCarter2000Book-1}
  If the interval $\Omega$ is a union of a finite number of pairwise disjoint intervals
$
\Omega=\Omega_{1} \cup \Omega_{2} \cup \cdots \cup \Omega_{m},
$
then
\begin{equation*}
\int_{\Omega} f \rd g=\sum_{j=1}^{m} \int_{\Omega_{j}} f \rd g
\end{equation*}
in the sense that if one side exists, then so does the other, and the two are equal. Moreover,   for any function $f$ defined at $\theta,$ then
 \begin{equation*}\label{BVtheta}
  \int_{[\theta, \theta]} f \rd  g =f(\theta)(g(\theta+)-g(\theta-)).
  \end{equation*}
\end{lmm}
\subsection{Formula of fractional integration by parts}
Recall the formula of integration by parts  involving
the Riemann-Stieltjes integrals  (cf. \!\cite[(1.20)]{Klebaner2005Book}):  {\em if  $f,g\in {\rm BV}(\bar\Omega)$, we have}
	\begin{align} \label{IPPW1122}
	\int_{a}^b f(x)  \, {\rm d} g(x) =\{f(x)g(x)\}\big|_{a^+}^{b^-} - \int_{a}^b  g(x)\,  {\rm d} f(x),
	\end{align}
where we denote
\begin{equation*}\label{uvpmlim}
	f(x)\big|_{a^+}^{b^-}=\lim\limits_{x\to b^-} f(x)- \lim\limits_{x\to a^+} f(x)=f(b-)-f(a+).
\end{equation*}	
In particular, if $f, g \in \mathrm{AC}(\bar{\Omega}),$ we have
\begin{equation*}\label{uvpmlim00}
\int_{a}^{b} f(x) g^{\prime}(x)\, \rd x+\int_{a}^{b} f^{\prime}(x) g(x)\, \rd x=\{f(x) g(x)\}\big|_{a} ^{b}.
\end{equation*}	
In what follows, we shall derive a formula of fractional integration parts from \eqref{IPPW1122} in a weaker  sense  than   the existing counterparts (cf.\! \cite{Samko1993Book,Bourdin2015ADE}).  For this purpose, we recap on  the definition of  the  Riemann-Liouville  fractional integral (cf.\!  \cite[p.\! 33, p.\! 44]{Samko1993Book}):  {\em  for any   $f\in L^1(\Omega),$  the left-sided and right-sided  Riemann-Liouville  fractional integrals of  real order $\rho \ge 0$ are defined  by}
	\begin{equation}\label{leftintRL}
	\begin{split}
	& (I_{a+}^\rho f) (x)=\frac 1 {\Gamma(\rho)}\int_{a}^x \frac{f(y)}{(x-y)^{1-\rho}} \rd y; \quad
	%\quad  a<x<b,
	(I_{b-}^\rho f) (x)=\frac 1 {\Gamma(\rho)}\int_{x}^b \frac{f(y)}{(y-x)^{1-\rho}} \rd y, %\;\;\; x\in \Omega,
	\end{split}
	\end{equation}
{\em for $x\in \Omega,$ where $\Gamma(\cdot)$ is the usual Gamma function. For $\mu \in (k-1, k]$ with $k \in \mathbb{N},$ the left-sided  and right-side Caputo fractional derivatives of order $\mu$ are respectively  defined by}
\begin{equation}\label{CaputoDev}
({ }^{C}\! D_{a+}^{\mu} f)(x)= (I_{a+}^{k-\mu} f^{(k)})(x); \quad ({ }^{C}\! D_{b-}^{\mu} f)(x)=(-1)^k (I_{b-}^{k-\mu} f^{(k)})(x).
\end{equation}

The following formula of fractional integration by parts plays an important role in the analysis, which can be derived from \eqref{IPPW1122}  (see  Appendix \ref{AppendixA}).
\begin{lmm}\label{FracIntPart-0} Let $\rho\ge 0, f(x)\in L^1(\Omega)$ and $ g(x)\in {\rm AC}(\bar \Omega).$
%	Given $\theta\in (-1,1),$ if  $u\in {\mathbb  W}^{m+s}_{\theta}(\Omega) $ with
	%$s\in (0,1]$ and integer  $m\ge 0$, we have the following estimates.
	\begin{itemize}
		\item[(i)] If $ I_{b-}^{\rho} f(x)\in {\rm BV}(\bar\Omega),$ then
	\begin{equation}\label{FracIntPart-1}
\begin{split}
\int^{b}_{a}   f(x) & \,I_{a+}^{\rho} g'(x)\,{\rm d}x  = \big\{g(x)\,  I_{b-}^{\rho} f(x)\big\}\big|_{a^+}^{b^-} - \int^{b}_{a} g(x) %\, ( I_{b-}^{1-s} f(x))'
\,{{\rm d}\big\{I_{b-}^{\rho} f(x)\big\}}. % {\mathcal D}_{\!b-}^{s} f(x)\,{\rm d}x,
\end{split}
\end{equation}
	\item[(ii)] If $ I_{a+}^{\rho} f(x)\in {\rm BV}(\bar\Omega),$ then
	\begin{equation}\label{FracIntPart-2}
\begin{split}
\int^{b}_{a}   f(x) & \,I_{b-}^{\rho} g'(x)\,{\rm d}x  = \big\{g(x)\,  I_{a+}^{\rho} f(x)\big\}\big|_{a^+}^{b^-} - \int^{b}_{a} g(x) %\, ( I_{b-}^{1-s} f(x))'
\,{{\rm d}\big\{I_{a+}^{\rho} f(x)\big\}}. % {\mathcal D}_{\!b-}^{s} f(x)\,{\rm d}x,
\end{split}
\end{equation}
	\end{itemize}
\end{lmm}
\begin{rem} {\em  If $\rho=0,$ then they reduce  \eqref{IPPW1122}.  It is known that the fractional integral can improve the regularity.  Indeed,  for $0<\rho<1$ and $u\in L^1(\Omega),$ we have $I_{a+}^{\rho} u, I_{b-}^{\rho} u\in L^p(\Omega)$ with $p\in [1, \rho^{-1})$ {\rm(cf.\! \cite[Prop.\! 2.1]{Bourdin2015ADE}).} } \qed
\end{rem}

%It is noteworthy that they are in a weaker sense than the existing ones (see, e.g., \cite{Samko1993Book,Bourdin2015ADE}), which can be derived from \eqref{IPPW1122}.
%%and the essential idea of the proof was actually scattered in \cite[Theorem 4.1]{Liu19MC-Optimal}.
% Here,

Compared with those in \cite{Samko1993Book,Bourdin2015ADE},  a weaker condition is imposed on $f(x)$ in  \eqref{FracIntPart-1}-\eqref{FracIntPart-2}, which turns out essential in dealing with the singular functions.  Moreover, for such functions,
 the limit values   $\lim_{x\to a^+}  I_{a+}^\rho f(x)$ in \eqref{FracIntPart-1},  and   $\lim_{x\to b^-}  I_{b-}^\rho f(x)$ in \eqref{FracIntPart-2}  might  be nonzero,  in contrast to a usual integral with $\rho=1.$ For example,  for $\rho\in (0,1),$ we have
\begin{equation*}\label{intformu007}
 I_{a+}^{1-\rho} \, (x-a)^{\rho-1}=
 I_{b-}^{1-\rho} \, (b-x)^{\rho-1}=  \Gamma(\rho),
\end{equation*}
which follow from the explicit formulas  $($cf.  \cite{Samko1993Book}$)$:   for  real $\eta>-1$ and $\rho\ge 0,$
\begin{equation}\label{intformu}
 I_{a+}^\rho (x-a)^\eta=  \dfrac{\Gamma(\eta+1)}{\Gamma(\eta+\rho+1)} (x-a)^{\eta+\rho};
 \quad I_{b-}^\rho  (b-x)^\eta=  \dfrac{\Gamma(\eta+1)}{\Gamma(\eta+\rho+1)} (b-x)^{\eta+\rho}.
% \quad
% {\mathcal D}_{\!a+}^s  \, (x-a)^\eta=\frac{\Gamma(\eta+1)}{\Gamma(\eta-s+1)} (x-a)^{\eta-s}.
\end{equation}
%We have similar formulas for right-sided Riemann-Liouville fractional integral/derivative of $(b-x)^\eta.$

In fact, we have the following more general formula,  which finds useful in  exemplifying some estimates to be presented later.
We  sketch the derivation in Appendix \ref{AppendixA0}.
\begin{prop}\label{boundaryvalue}  Let $f(x)=(x-a)^{\gamma}g(x)$ with real $\gamma>-1$, where $g(x)$  is bounded and Riemann integrable on $[a,a+\delta)$ for some $\delta>0.$  Then  for real  $\rho >0$,  we have % with $\mu\ge -\alpha$,
	\begin{equation}\label{integbA}
	\lim_{x\to a^+}(I_{a+}^\rho\, f) (x)=
	\begin{cases}
	0,  & {\rm if}\;\;  \rho >-\gamma, \\
	g(a)\Gamma(\gamma+1),&{\rm if}\;\; \rho=-\gamma,\\
	\infty, &{\rm if}\;\; \rho<-\gamma.
	\end{cases}
	\end{equation}
Let  $f(x)=(b-x)^{\gamma}g(x)$, $\gamma>-1$, and  $g(x)$ be bounded and Riemann integrable on $(b-\delta,b].$ Then the same result holds for the limit $\lim\limits_{x\to b^-}(I_{b-}^\rho\, f) (x)$ but with $g(b)$ in place of $g(a).$
 \end{prop}

 \subsection{Fractional Taylor formula} Needless to say, the Taylor formula plays a fundamental role in many branches of mathematics.
For comparison purpose, we recall this well-known formula:  {\em Let $k \geq 1$ be an integer and let $f(x)$ be a real function that is  $k$ times differentiable at the point $x=\theta$.  Further, let $f^{(k)}(x)$ be absolutely continuous on the closed interval between $\theta$ and $x$.  Then we have}
 \begin{equation}\label{TaylorA}
 f(x)= \sum_{j=0}^{k} \frac{f^{(j)}(\theta)}{j!}(x-\theta)^{j}  %f(\theta)+f^{\prime}(\theta)(x-\theta) +\frac{f^{\prime \prime}(\theta)}{2 !}(x-\theta)^{2}+\cdots+\frac{f^{(k)}(\theta)}{k !}(x-\theta)^{k}
 +
 \int_{\theta}^{x} \frac{f^{(k+1)}(t)}{k !}(x-t)^{k}\, \rd t.
 \end{equation}
%where the Taylor polynomial is
% \begin{equation}\label{TaylorA}
% T_k(x):=T_k(x;\theta)=f(\theta)+f^{\prime}(\theta)(x-\theta) +\frac{f^{\prime \prime}(\theta)}{2 !}(x-\theta)^{2}
% +\cdots+\frac{f^{(k)}(\theta)}{k !}(x-\theta)^{k}.
% \end{equation}
 Note that since $f^{(k)}(x)$ is an AC-function, $f^{(k+1)}(x)$ exists as an $L^1$-function.
 %Hereafter, we denote the Taylor polynomial by $T_k(x;a).$

As a second building block for the analysis, we  derive a fractional  Taylor  formula from Lemma  \ref{FracIntPart-0} and
 \eqref{TaylorA}.
\begin{thm}[{\bf Fractional Taylor formula}]\label{NewFTFV1}  Let  $\mu \in (k-1, k]$ with $k \in \mathbb{N},$ and  let $f(x)$ be a real function that is  $(k-1)$times differentiable at the point $x=\theta$.
\begin{itemize}
 \item[(i)] If $f^{(k-1)}\in {\rm AC}([\theta, x])$ and    ${ }^{C}\! D_{\theta+}^{\mu} f \in  {\rm BV}([\theta, x]),$ then we have the left-sided fractional Taylor formula
\begin{equation}\begin{split}\label{LetfTaylor-00}
f(x)=\sum_{j=0}^{k-1} \frac{f^{(j)}(\theta)}{j!}(x-\theta)^{j} + \frac{{ }^{C}\! D_{\theta+}^{\mu} f(\theta+)}{\Gamma(\mu+1)}{(x-\theta)^{\mu}}+  \frac{1}{\Gamma(\mu+1)}
 \int_{\theta}^{x} (x-t)^{\mu}\, \rd\{{ }^{C}\! D_{\theta+}^{\mu} f(t)\}.
\end{split}\end{equation}
\item[(ii)] If $f^{(k-1)}\in {\rm AC}([x,\theta])$ and    ${ }^{C}\! D_{\theta-}^{\mu} f \in  {\rm BV}([x,\theta]),$ then we have the right-sided fractional Taylor formula
\begin{equation}\begin{split}\label{LetfTaylor-000}
f(x)=\sum_{j=0}^{k-1} \frac{f^{(j)}(\theta)}{j!}(x-\theta)^{j} + \frac{{ }^{C}\! D_{\theta-}^{\mu} f(\theta-)}{\Gamma(\mu+1)}{(\theta-x)^{\mu}}
-\frac{1}{\Gamma(\mu+1)} \int^{\theta}_{x} (t-x)^{\mu}\, \rd \{{ }^{C}\! D_{\theta-}^{\mu} f(t)\}.
\end{split}\end{equation}
\end{itemize}
\end{thm}
\begin{proof} By \eqref{TaylorA}  (with  $k\to k-1$), we have
\begin{equation}\begin{split}\label{LetfTaylor-1}
f(x)&=\sum_{j=0}^{k-1} \frac{f^{(j)}(\theta)}{j!}(x-\theta)^{j}  +\frac{1}{(k-1)!} \int_{\theta}^{x}(x-t)^{k-1}f^{(k)}(t)\, \rd t.
\end{split}\end{equation}
From \eqref{intformu}, we find readily that  for $x>t,$
$$(x-t)^{k-1} =- \frac{(k-1)!}{\Gamma(\mu+1)}I_{x-}^{k-\mu}\Big\{\frac{\rd}{\rd t} (x-t)^{\mu}\Big\}.$$
Thus, we  can rewrite \eqref{LetfTaylor-1} as  %have %obtain for $x>\theta,$
\begin{equation}\begin{split}\label{LetfTaylor-10}
f(x)&=\sum_{j=0}^{k-1} \frac{f^{(j)}(\theta)}{j!}(x-\theta)^{j}   - \frac{1}{\Gamma(\mu+1)} \int_{\theta}^{x}I_{x-}^{k-\mu}\Big\{\frac{\rd}{\rd t} (x-t)^{\mu}\Big\}f^{(k)}(t)\, \rd t.
%+\frac{1}{(k-1)!} \int_{\theta}^{x}(x-t)^{k-1}f^{(k)}(t)\, \rd t.
\end{split}\end{equation}
% \begin{equation}\begin{split}\label{LetfTaylor-2}
%\frac{1}{(k-1)!} \int_{\theta}^{x}(x-t)^{k-1} f^{(k)}(t)\, \rd t =& - \frac{1}{\Gamma(\mu+1)} \int_{\theta}^{x}I_{x-}^{k-\mu}\Big\{\frac{\rd}{\rd t} (x-t)^{\mu}\Big\}f^{(k)}(t) \rd t.
%  \end{split}\end{equation}
  Substituting $a,$ $b,$ $\rho,$  $f$ and $g$ in \eqref{FracIntPart-2} of Lemma \ref{FracIntPart-0}  by $\theta,$ $x,$ $k-\mu$, $f^{(k)}(t)$ and $(x-t)^\mu,$ respectively, we obtain that  for $x>\theta,$
   \begin{equation}\begin{split}\label{LetfTaylor-2+1}
\int_{\theta}^{x}I_{x-}^{k-\mu}\Big\{\frac{\rd}{\rd t} (x-t)^{\mu}\Big\}f^{(k)}(t)\, \rd t
&=-{I_{\theta+}^{k-\mu} f^{(k)}(\theta+)}(x-\theta)^{\mu}-\int_{\theta}^{x} (x-t)^{\mu}\, \rd \{{I_{\theta+}^{k-\mu} f^{(k)}(t)}\} \\
&=-{ }^{C}\! D_{\theta+}^{\mu} f(\theta+){(x-\theta)^{\mu}}-\int_{\theta}^{x} (x-t)^{\mu}\, \rd \{{ }^{C}\! D_{\theta+}^{\mu} f(t)\},
  \end{split}\end{equation}
  where in the last step, we used the definition  \eqref{CaputoDev}.   Thus, we obtain  \eqref{LetfTaylor-00} from   \eqref{LetfTaylor-10}-\eqref{LetfTaylor-2+1} immediately.

  The right-sided formula \eqref{LetfTaylor-000} can be obtained in a very similar fashion.
\end{proof}

\begin{rem}\label{existing} {\em When $\mu=k,$ the fractional Taylor formulas  \eqref{LetfTaylor-00} and \eqref{LetfTaylor-000}
lead to  \eqref{TaylorA}.
The fractional formula can be viewed as the ``interpolation" of  the integer-order Taylor  formulas with the regularity indexes  $k-1$ and $k.$ Apparently,  the integer-order Taylor formula \eqref{TaylorA} is exact for all $f\in {\mathcal P_k}={\rm span}
\{(x-\theta)^j\,:\, 0\le j\le k\}.$ In the fractional case, the exactness of \eqref{LetfTaylor-00}  is for all $f\in {\mathcal P_{k-1}}\cup\{(x-\theta)^\mu\}$ {\rm(}i.e., the remainder vanishes{\rm).} We can verify this readily from \eqref{CaputoDev} and the fundamental formula:
$ { }^{C}\! D_{\theta+}^{\mu} \{(x-\theta)^\mu\}=\Gamma({\mu+1}).$ Note that the right-sided formula \eqref{LetfTaylor-000} is exact for all  $f\in {\mathcal P_{k-1}}\cup\{(\theta-x)^\mu\}.$} \qed
\end{rem}

We remark that there are several versions of fractional Taylor formulas for functions with different regularities.  For example,
Anastassiou \cite[(21)]{Anastassiou2009Chaos-On} stated the right-sided fractional Taylor formula:  for real $\mu\ge 1,$  let $k=[\mu]$ be its integer part, and assume that $f,f',\ldots,f^{(k-1)}\in {\rm AC}([x,\theta]).$ Then
\begin{equation*}\label{Anast}
f(x)=\sum_{j=0}^{k-1} \frac{f^{(j)}(\theta)}{j !}(x-\theta)^{j}+\frac{1}{\Gamma(\mu)} \int_{x}^{\theta}(t-x)^{\mu-1} \, { }^{C}\!D_{\theta{-}}^{\mu} f(t)\, \rd t.
\end{equation*}
Kolwankar and Gangal   \cite{Kolwankar1998PRL-Local}  %\cite{Kolwankar1996C-Fractional,Kolwankar1998PRL-Local}
presented  some  local fractional Taylor
expansion with a  different fractional derivative in the remainder.

\section{Legendre expansions of functions with interior singularities}\label{sect3main}
\setcounter{equation}{0}
\setcounter{lmm}{0}
\setcounter{thm}{0}

It is known that much of the error analysis for orthogonal polynomial approximation and associated interpolation and quadrature  relies on  the decay rate of the expansion coefficient (cf.\! \cite{Xiang2012SINUM,Majidian2017ANM}).
Remarkably, we find that  the  spirit in  deriving the fractional Taylor formula in Theorem \ref{NewFTFV1}  can be extended to obtain  an analogous  formula for  the Legendre expansion coefficient
\begin{equation}\label{ancoef}
\hat u_n^L=\frac{2n+1}{2}\int_{-1}^{1}{u(x)P_n(x)}\, {\rm d} x,
\end{equation}
where $P_n(x)$ is the Legendre polynomial of degree $n.$
This formula lays the groundwork for all the forthcoming analysis. In fact, the argument  is also different from that for  the Chebyshev expansion coefficient  in \cite{Trefethen2008SIREV,Trefethen2013Book,Majidian2017ANM,Liu19MC-Optimal}.
% largely due to the different weight functions for these two classes of orthogonal polynomials. %fact,  here we can take the advantage of the uniform Legendre weight function $\omega=1.$

%{\color{red} Please clear up the new proof here for the expansion coefficients, where we cite the formulas of GGF-Fs for the case we used! We also use Caputo derivative!}

%Using the fractional Taylor's formula, we give the new proof for Theorem \ref{IdentityForUn}.

\subsection{Fractional formula for the Legendre expansion coefficient}\label{sub31}
%In what follows,  we estimate  the error  between $u$ of AC$_{m}$-BV$_{s,\theta}$-regularity  and
%its finite Legendre expansion $\pi_N^L  u$ as follows
%%
%%between $L^2$-orthogonal projection
%% For  $ u\in L^2(\Omega)$, we expand it in Legendre series and denote the  partial sum by
% % (see, for instance, \cite{Wang12MC-convergence,Wang16SJNA-optimal,Xiang12SJNA-error,Wang18AML-new,Schwab1998Book})
%\begin{equation}\label{Cbexp}
%u(x)=\sum_{n=0}^{\infty}\hat u_n^L\, P_n(x),\quad\;\;   (\pi_N^L  u)(x)= \sum_{n=0}^N \hat u_n^L\, P_n(x),
%\end{equation}
%where $P_n(x)$ is the usual Legendre polynomial of degree $n$, and
%%\quad \pi_N^L  u(x)=\sum_{n=0}^{N}\hat u_n^L\, P_n(x),
%The following representation of the Legendre expansion is the starting point of the analysis.
In what follows,  we assume that $u$ has a limited regularity with an interior singularity at $\theta\in (-1,1),$  e.g.,
$u(x)=|x-\theta|^\alpha$ with $\alpha>-1.$   Note   that
the results   can be extended to multiple interior singularities straightforwardly.
\begin{theorem}\label{IdentityForUn}
Let  $\mu \in (k-1, k]$ with $k \in \mathbb{N}$ and let $\theta\in(-1,1).$  If  $u,u',\ldots,u^{(k-1)}\in {\rm AC}([-1,1]),$  ${ }^{C}\! D_{\theta+}^{\mu} u \in  {\rm BV}([\theta, 1])$ and
  ${ }^{C}\! D_{\theta-}^{\mu} u \in  {\rm BV}([-1,\theta]),$  then  we have the following  representation of the Legendre expansion coefficient for each $n\ge\mu+1,$
\begin{equation}\label{HatUnCaseC}
\begin{split}
\hat u_n^L = \frac{2n+1}{2} &\Big\{ (I_{1-}^{\mu+1} P_n)(\theta) ({{ }^{C}\! D_{\theta+}^{\mu} u)(\theta+)}
+\int_{\theta}^{1} (I_{1-}^{\mu+1} P_n)(x) \,  \rd \big\{{ }^{C}\! D_{\theta+}^{\mu} u(x)\big\}\\
&\quad  +(I_{-1+}^{\mu+1} P_n)(\theta) ({{ }^{C}\! D_{\theta-}^{\mu} u)(\theta-)}  - \int^{\theta}_{-1}(I_{-1+}^{\mu+1} P_n)(x)\, \rd \big\{{ }^{C}\! D_{\theta-}^{\mu} u(x)\big\}\Big\},
%\hat u_n^L= &\frac{2n+1}{2} \bigg\{ %(-1)^{ n+[ n-s]}
%{}^{+}{\mathcal G}_n^{(\mu)}(\theta)\, { }^{C}\! D_{\theta+}^{\mu} u(\theta+)+ \int_\theta^1    {}^+{\mathcal G}_n^{(\mu)}(x)  \, {\rd\,\big\{ { }^{C}\! D_{\theta+}^{\mu} u(x)}\big\}\\&
%  \quad
%\!\!+\,{}^-{\mathcal G}_n^{(\mu)}(\theta)\, { }^{C}\! D_{\theta-}^{\mu} u(\theta-) -\int_{-1}^\theta    {}^-{\mathcal G}_n^{(\mu)}(x)  \, {\rd\, \big\{ { }^{C}\! D_{\theta-}^{\mu} u(x)}\big\}\bigg\},
\end{split}
\end{equation}
where the fractional integrals of $P_n(x)$ can be evaluated explicitly  by
\begin{equation}\label{twoGGFs}
\begin{split}
&(I_{1-}^{\mu+1} P_n)(x)= %\frac{1}{\Gamma(\mu+1)} \int_{x}^{1} {(y-x)^{\mu}}\,P_n(y)\, {\rm d} y=
\frac{(1-x)^{\mu+1}}{\Gamma(\mu+2)} \frac{P_{n}^{(\mu+1, -\mu-1)}(x)}{P_{n}^{(\mu+1,-\mu-1)}(1)},\\
& (I_{-1+}^{\mu+1} P_n)(x) %=\frac{1}{\Gamma(\mu+1)} \int_{-1}^{x} {(x-y)^{\mu}}\,P_n(y)\, {\rm d} y
=\frac{(1+x)^{\mu+1}}{\Gamma(\mu+2)} \frac{P_{n}^{( -\mu-1,\mu+1)}(x)}{P_{n}^{(\mu+1,-\mu-1)}(1)}.
%\quad  P_{n}^{(\mu+1,-\mu-1)}(1)=\frac{\Gamma(n+\mu+2)} {n!\Gamma(\mu+2)}.
\end{split}
%{}^{\pm}{\mathcal G}_n^{(\mu)}(x):=(1 \mp x)^{\mu+1} \frac{P_{n}^{(\pm(\mu+1), \mp(\mu+1))}(x)}{P_{n}^{(\pm(\mu+1), \mp(\mu+1))}(1)}.
% ;\quad v_{\pm}(x)= (I_{\theta\pm}^{1-s} u^{(m)})(x)
%{}^r{\mathcal G}_n^{(\sigma)}(x):=(1-x^2)^\sigma   \, {}^{r}G_{n-\sigma}^{(\sigma+1/2)}(x).
\end{equation}
Here $P_{n}^{(\mu+1, -\mu-1)}(x)$ and $P_{n}^{( -\mu-1,\mu+1)}(x)$ are the generalised Jacobi polynomials defined by the hypergeometric functions  as in  Szeg\"o \cite[p.\! 64]{Szego1975Book}.
%\begin{equation}\label{HatUnCaseC}
%\begin{split}
%\hat u_n^L= \frac{2n+1}{2^{\sigma+1}\Gamma(\sigma+1)} \bigg\{& %(-1)^{ n+[ n-s]}
%-\int_{-1}^\theta   (1-x^2)^\sigma\,   {}^{l}G_{n-\sigma}^{(\sigma+1/2)}(x) \, {\rd \big\{I_{\theta-}^{1-s}u^{(m)}  (x)\big\}}
%\\&
%- \lim_{x\to \theta^-}  \left\{ (1-x^2)^\sigma   \, {}^{l}G_{n-\sigma}^{(\sigma+1/2)}(x)\, I_{\theta-}^{1-s}u^{(m)} (x)\right\}\\
%&+\int^{1}_\theta    (1-x^2)^\sigma\,   {}^{r\!}G_{n-\sigma}^{(\sigma+1/2)}(x) \,  {\rd\big\{I_{\theta+}^{1-s}u^{(m)}  (x)\big\}}\\
%&+\lim_{x\to \theta^+}\left\{ (1-x^2)^\sigma \, {}^{r\!}G_{n-\sigma}^{(\sigma+1/2)}(x)  \, I_{\theta+}^{1-s}u^{(m)}  (x) \right\}\bigg\}.
%\end{split}
%\end{equation}
% Moreover we have the following upper bounds
%\begin{equation}\label{BoundForUnRealB}
%	\begin{split}
%	|\hat u_n^{L}|\le & \frac{U_\theta^{m,s}}{2^{\sigma+1} \sqrt{\pi}}
%	\frac{ (2n+1)\Gamma( ({n-\sigma}+1)/ 2)}{\Gamma( ({n+\sigma})/2+1)} .
%	\end{split}
%	\end{equation}
\end{theorem}
\begin{proof} Given the regularity of $u$, we  obtain from  the fractional Taylor formulas in Theorem \ref{NewFTFV1} that
for $x\in (\theta,1),$
\begin{equation}\begin{split}\label{IdentityForUn-1}
u(x)=\sum_{j=0}^{k-1} \frac{u^{(j)}(\theta)}{j!}(x-\theta)^{j} + \frac{{ }^{C}\! D_{\theta+}^{\mu} u(\theta+)}{\Gamma(\mu+1)}{(x-\theta)^{\mu}}+  \frac{1}{\Gamma(\mu+1)}
 \int_{\theta}^{x} (x-t)^{\mu}\, \rd \big\{{ }^{C}\! D_{\theta+}^{\mu} u(t)\big\},
\end{split}\end{equation}
and for $x\in (-1,\theta),$
\begin{equation}\begin{split}\label{IdentityForUn-2}
u(x)=\sum_{j=0}^{k-1} \frac{u^{(j)}(\theta)}{j!}(x-\theta)^{j} + \frac{{ }^{C}\! D_{\theta-}^{\mu} u(\theta-)}{\Gamma(\mu+1)}{(\theta-x)^{\mu}}
-\frac{1}{\Gamma(\mu+1)} \int^{\theta}_{x} (t-x)^{\mu}\, \rd \big\{{ }^{C}\! D_{\theta-}^{\mu} u(t)\big\}.
\end{split}\end{equation}
Substituting  \eqref{IdentityForUn-1} and  \eqref{IdentityForUn-2} into \eqref{ancoef} leads to
\begin{equation}\label{IdentityForUn-3}
\begin{split}
&\frac {2\,  \hat u_n^L}{2n+1}  =  \int_{-1}^{1}{u(x)P_n(x)}\, {\rm d} x= \sum_{j=0}^{k-1} \frac{u^{(j)}(\theta)}{j!}
\int_{-1}^{1}(x-\theta)^{j} {P_n(x)}\, {\rm d} x \\
&\;\;  +\frac {{ }^{C}\! D_{\theta+}^{\mu} u(\theta+)}{\Gamma(\mu+1)}\int_{\theta}^{1} {(x-\theta)^{\mu}}\,P_n(x)\, {\rm d} x
+\frac{1}{\Gamma(\mu+1)} \int_{\theta}^{1}\Big(\int_{\theta}^{x} (x-t)^{\mu} \rd \{{ }^{C}\! D_{\theta+}^{\mu} u(t)\}\Big)\,P_n(x)\, {\rm d} x\\
&\;\; +\frac{{ }^{C}\! D_{\theta-}^{\mu} u(\theta-)} {\Gamma(\mu+1)}  \int_{-1}^{\theta} {(\theta-x)^{\mu}}\,P_n(x)\, {\rm d} x-\frac{1}{\Gamma(\mu+1)} \int_{-1}^{\theta}\Big(\int_{x}^{\theta} (t-x)^{\mu} \rd \{{ }^{C}\! D_{\theta-}^{\mu} u(t)\}\Big)\,P_n(x)\, {\rm d} x.
\end{split}\end{equation} %\quad \pi_N^L  u(x)=\sum_{n=0}^{N}\hat u_n^L\, P_n(x),
From the orthogonality of the Legendre polynomials, we obtain that  for $n\ge \mu+1\ge k,$
\begin{equation}\label{IdentityForUn-4}
\begin{split}
\int_{-1}^{1}(x-\theta)^{j} P_n(x) {\rm d} x=0,\quad 0\le j\le k-1.
\end{split}\end{equation} %\quad \pi_N^L  u(x)=\sum_{n=0}^{N}\hat u_n^L\, P_n(x),
We find readily that for a fixed $\theta\in (-1,1),$
\begin{equation}\label{IdentityForUn-60}
\begin{split}
&\int_{\theta}^{1}\Big(\int_{\theta}^{x} {(x-t)^{\mu}}\, \rd \{{ }^{C}\! D_{\theta+}^{\mu} u(t)\}\Big)\,P_n(x)\, {\rm d} x=
\int_{\theta}^{1}\Big(\int_{t}^{1} {(x-t)^{\mu}} \,P_n(x) {\rm d} x\Big) \rd \{{ }^{C}\! D_{\theta+}^{\mu} u(t)\},
\end{split}\end{equation}
and
\begin{equation}\label{IdentityForUn-602}
\begin{split}
& \int_{-1}^{\theta}\Big(\int_{x}^{\theta} (t-x)^{\mu} \rd \{{ }^{C}\! D_{\theta-}^{\mu} u(t)\}\Big)\,P_n(x)\, {\rm d} x=
\int^{\theta}_{-1}\Big(\int_{-1}^{t} {(t-x)^{\mu}} \,P_n(x) \,{\rm d} x\Big)\, \rd \{{ }^{C}\! D_{\theta-}^{\mu} u(t)\}.
\end{split}\end{equation}
In view of the definition of the fractional integral in \eqref{leftintRL} and \eqref{IdentityForUn-4}-\eqref{IdentityForUn-602}, we can rewrite
\eqref{IdentityForUn-3} as
\begin{equation}\label{IdentityForUn-3A}
\begin{split}
\frac {2\,  \hat u_n^L}{2n+1}  &=
(I_{1-}^{\mu+1} P_n)(\theta)  ({{ }^{C}\! D_{\theta+}^{\mu} u)(\theta+)}
+\int_{\theta}^{1} (I_{1-}^{\mu+1} P_n)(t) \,  \rd \{{ }^{C}\! D_{\theta+}^{\mu} u(t)\}\\
&\quad  +(I_{-1+}^{\mu+1} P_n)(\theta) ({ }^{C}\! D_{\theta-}^{\mu} u)(\theta-) - \int^{\theta}_{-1}(I_{-1+}^{\mu+1} P_n)(t)\, \rd \{{ }^{C}\! D_{\theta-}^{\mu} u(t)\},
\end{split}\end{equation}
which yields \eqref{HatUnCaseC}.  The two fractional integral identities of $P_n(x)$
in \eqref{twoGGFs} can be obtained from the  formulas  of the Jacobi polynomials (cf.\! Szeg\"o  \cite[p.\! 96]{Szego1975Book}), due to the Bateman's fractional integral formula (cf.\! \cite{Andrews1999Book}).  This ends the proof.
\end{proof}

We see from the above proof that the identity \eqref{HatUnCaseC}  is rooted in the fractional Taylor formula in Theorem \ref{NewFTFV1}. Also note that
 when $\mu=k$, the formula  \eqref{IdentityForUn} takes a much simpler form. Firstly,  the AC-BV regularity reduces to the setting considered by Trefethen \cite{Trefethen2008SIREV,Trefethen2013Book},  Xiang and Bornemann \cite{Xiang2012SINUM} among others (where one motivative example for the framework therein is to best characterise the regularity of  $u(x)=|x|$).  Secondly,  from
  Szeg\"o  \cite[Chap.\!\! 4]{Szego1975Book}, we find that for $\mu>-2,n\ge 0,$
  \begin{equation}\label{Pnmumu}
  P_{n}^{(\mu+1,-\mu-1)}(1)=\frac{\Gamma(n+\mu+2)} {n!\,\Gamma(\mu+2)},
  \end{equation}
  and for $n\ge k+1,$
  \begin{equation}\label{Pnmumu2}
  P_{n}^{(-k-1,k+1)}(x)= \frac{(n-k-1)! (n+k+1)!} {(n!)^2}\Big(\frac{x-1} 2\Big)^{k+1} P_{n-k-1}^{(k+1,k+1)}(x).
  \end{equation}
  Thus,  we can rewrite the second formula in \eqref{twoGGFs} with $\mu=k$  in terms of the usual Jacobi polynomial as follows
  \begin{equation}\label{Pnmumu3}
(I_{-1+}^{k+1} P_n)(x) = \frac{(-1)^{k+1}\,(n-k-1)!} {2^{k+1}\,n!} (1-x^2)^{k+1} P_{n-k-1}^{(k+1,k+1)}(x).
  \end{equation}
Following the same lines as above and using the parity of Jacobi polynomials, we can  reformulate the first formula  in \eqref{twoGGFs} with $\mu=k$ as
 \begin{equation}\label{Pnmumu4}
(I_{1-}^{k+1} P_n)(x) =  \frac{(n-k-1)!} {2^{k+1}\,n!} (1-x^2)^{k+1} P_{n-k-1}^{(k+1,k+1)}(x)= (-1)^{k+1}\, (I_{-1+}^{k+1} P_n)(x).
  \end{equation}
 In view of this relation, we find from    \eqref{CaputoDev} with $\mu=k$ that
%\begin{equation*}\label{CaputoDev-V2}
%({ }^{C}\! D_{\theta-}^{k} u)(\theta-) =(-1)^ku^{(k)}(\theta-),\quad ({ }^{C}\! D_{\theta+}^{k} u)(\theta+) =u^{(k)}(\theta+),
%\end{equation*}
 \eqref{HatUnCaseC} reduces to
  \begin{equation}\label{HatUnCaseC-V2+1}
\begin{split}
\hat u_n^L & = \frac{2n+1}{2} \Big\{
  u^{(k)}(\theta+) (I_{1-}^{k+1} P_n)(\theta)
+\int_{\theta}^{1} (I_{1-}^{k+1} P_n)(x) \,  \rd \big\{u^{(k)}(x)\big\}\\
&\quad -u^{(k)}(\theta-) (I_{1-}^{k+1} P_n)(\theta) + \int^{\theta}_{-1}(I_{1-}^{k+1} P_n)(x)\, \rd \big\{ u^{(k)}(x)\big\}\Big\}\\
&= \frac{2n+1}{2} \bigg\{ \int^{\theta}_{-1}+ \int_{[\theta,\theta]} +\int_{\theta}^{1} \bigg\} (I_{1-}^{k+1} P_n)(x)\, \rd \big\{ u^{(k)}(x)\big\}.
\end{split}
\end{equation}
By virtue of the splitting rule in Lemma \ref{LmmCarter2000Book-1},  we can summarise the formula of  the Legendre expansion coefficient with $\mu=k$ as follows.
\begin{cor}\label{sequal1}
 If $u,u',\ldots,u^{(k-1)}\in {\rm AC}([-1,1])$  and $ u^{(k)} \in  {\rm BV}([-1, 1])$ with $k \in \mathbb{N},$
 then  we have  for  all $n\ge k+1,$
\begin{equation}\label{HatUnCases1}
\begin{split}
\hat u_n^L= \frac{2n+1}{2} &\int^{1}_{-1}(I_{1-}^{k+1} P_n)(x)\, \rd \{ u^{(k)}(x)\},
\end{split}
\end{equation}
where $(I_{1-}^{k+1} P_n)(x)$ can be explicitly evaluated by \eqref{Pnmumu4}.
\end{cor}
%\begin{proof} It is evident that by the definition
%From the parity of the generalised Jacobi polynomial (cf. Szeg\"o  \cite{Szego1975Book}):
%
%{twoGGFs}
%
%
%
%Using  \eqref{leftintRL}, \eqref{BoundIntLeg-1}, \eqref{BoundIntLeg-2}, \eqref{BoundIntLeg-4}, \eqref{BoundIntLeg-5} and $P_{n-k-1}^{(k+1,k+1)}(1)=(-1)^{n-k-1}P_{n-k-1}^{(k+1,k+1)}(-1),$ yields
%and
%\begin{equation}\label{BoundIntLeg-4V2}
%\begin{split}
%&(I_{1-}^{k+1} P_n)(x)=\frac{(1-x^2)^{k+1}}{2^{k+1}(k+1)!}\frac{P_{n-k-1}^{(k+1,k+1)}(x)}{P_{n-k-1}^{(k+1,k+1)}(1)}=(-1)^{k+1}(I_{-1+}^{k+1} P_n)(x).
%\end{split}
%%{}^{\pm}{\mathcal G}_n^{(\mu)}(x):=(1 \mp x)^{\mu+1} \frac{P_{n}^{(\pm(\mu+1), \mp(\mu+1))}(x)}{P_{n}^{(\pm(\mu+1), \mp(\mu+1))}(1)}.
%% ;\quad v_{\pm}(x)= (I_{\theta\pm}^{1-s} u^{(m)})(x)
%%{}^r{\mathcal G}_n^{(\sigma)}(x):=(1-x^2)^\sigma   \, {}^{r}G_{n-\sigma}^{(\sigma+1/2)}(x).
%\end{equation}
%Taking $\mu=k$ in \eqref{HatUnCaseC}, with  \eqref{CaputoDev-V2} and \eqref{BoundIntLeg-4V2}, we have
%Together with Lemmas \eqref{LmmCarter2000Book-1} and \eqref{LmmCarter2000Book-2}, we obtain \eqref{HatUnCases1}.
%\end{proof}

It is seen from Theorem \ref{IdentityForUn} that the decay rate of  $\hat u_n^L$ for $u(x)$  with a fixed regularity index $\mu$ is determined by the fractional integrals of $P_n(x).$  Indeed, we have the following bound.
\begin{lmm}\label{BoundIntLeg}
For $\mu>-1/2$ and $n\ge \mu+1,$ we have
\begin{equation}\label{BoundIntLeg-0}
\max_{|x|\le 1}\Big\{\big|(I_{1-}^{\mu+1} P_n)(x)\big|, \big|(I_{-1+}^{\mu+1} P_n)(x)\big|\Big\}\le  \frac{1}{2^{\mu+1} \sqrt \pi }
	\frac{ \Gamma( (n-\mu)/ 2)}{\Gamma( (n+\mu+3)/2)}.
\end{equation}
\end{lmm}
\begin{proof} According to Szeg\"o \cite[p.\! 62]{Szego1975Book},  the generalised Jacobi polynomials  with real parameters $\alpha,\beta$ are defined by the hypergeometric functions as
\begin{equation}\label{BoundIntLeg-1}
P_n^{(\alpha,\beta)}(x) =\frac{(\alpha+1)_{n}}{n !}\, {}_2F_1\Big(\!\!-n, n+\alpha+\beta+1;\alpha+1;\frac{1-x} 2\Big),\quad x\in (-1,1),
\end{equation}
or alternatively,
\begin{equation}\label{BoundIntLeg-2}
P_{n}^{(\alpha, \beta)}(x)=(-1)^{n} \frac{(\beta+1)_{n}}{n !}{ }_{2} F_{1}\Big(\!-n, n+\alpha+\beta+1 ; \beta+1 ; \frac{1+x}{2}\Big),\quad x\in (-1,1).
\end{equation}
Recall  the Euler transform identity (cf. \cite[p.\! 95]{Andrews1999Book}): for $a,b,c \in {\mathbb R}$ and $-c\not \in {\mathbb N}_0,$
\begin{align}\label{Euler}
{}_2F_1(a,b;c;z) =(1-z)^{c-a-b}{}_2F_1(c-a,c-b;c;z),\quad |z|<1.
\end{align}
Taking $a=-n, b=n+\alpha+\beta+1, c=\alpha+1$ and $z=(1-x)/2$ in \eqref{Euler},  we obtain
 \begin{equation}\label{BoundIntLeg-3}
 \begin{split}
{}_2F_1\Big(\!\!-n, n+\alpha+\beta+1;\alpha+1;\frac{1-x} 2\Big)& =\Big(\frac {1+x} 2\Big)^{-\beta}\!
{}_2F_1\Big(\!\!\,n+\alpha+1, -n-\beta;\alpha+1;\frac{1-x} 2\Big)\\
&=\Big(\frac {1+x} 2\Big)^{-\beta}\!
{}_2F_1\Big(\!\! -n-\beta, n+\alpha+1;\alpha+1;\frac{1-x} 2\Big).
%\\
%=&\Big(\frac {1+x} 2\Big)^{-\beta}{}^{r\!}P_{n+\beta}^{(\alpha,-\beta)}(x).
   \end{split}
\end{equation}
%where we used the obvious fact: ${}_2F_1(a,b;c;z)={}_2F_1(b,a;c;z).$
From \eqref{twoGGFs}, \eqref{BoundIntLeg-1}, \eqref{BoundIntLeg-3} and  ${}_2F_1(a,b;c;0)=1,$ we get
\begin{equation}\label{BoundIntLeg-4}
\begin{split}
(I_{1-}^{\mu+1} P_n)(x)&=
\frac{(1-x)^{\mu+1}}{\Gamma(\mu+2)} \frac{P_{n}^{(\mu+1, -\mu-1)}(x)}{P_{n}^{(\mu+1,-\mu-1)}(1)}=\frac{(1-x)^{\mu+1}}{\Gamma(\mu+2)}\, {}_2F_1\Big(\!\!-n, n+1;\mu+2;\frac{1-x} 2\Big)\\
&=\frac{(1-x^2)^{\mu+1}}{2^{\mu+1}\Gamma(\mu+2)}
{}_2F_1\Big(\!\! -n+\mu+1, n+\mu+2;\mu+2;\frac{1-x} 2\Big).
\end{split}
%{}^{\pm}{\mathcal G}_n^{(\mu)}(x):=(1 \mp x)^{\mu+1} \frac{P_{n}^{(\pm(\mu+1), \mp(\mu+1))}(x)}{P_{n}^{(\pm(\mu+1), \mp(\mu+1))}(1)}.
% ;\quad v_{\pm}(x)= (I_{\theta\pm}^{1-s} u^{(m)})(x)
%{}^r{\mathcal G}_n^{(\sigma)}(x):=(1-x^2)^\sigma   \, {}^{r}G_{n-\sigma}^{(\sigma+1/2)}(x).
\end{equation}
Similarly, we can show that %using \eqref{twoGGFs},  \eqref{BoundIntLeg-1}, \eqref{BoundIntLeg-2}, \eqref{Euler} and ${}_2F_1(a,b;c;0)=1,$ we  obtain
\begin{equation}\label{BoundIntLeg-5}
\begin{split}
(I_{-1+}^{\mu+1} P_n)(x)
%&=
%\frac{(1+x)^{\mu+1}}{\Gamma(\mu+2)} \frac{P_{n}^{( -\mu-1,\mu+1)}(x)}{P_{n}^{(\mu+1,-\mu-1)}(1)}=\frac{(-1)^n(1+x)^{\mu+1}}{\Gamma(\mu+2)}\, {}_2F_1\Big(\!\!-n, n+1;\mu+2;\frac{1+x} 2\Big)\\
&=\frac{(-1)^n (1-x^2)^{\mu+1}}{2^{\mu+1}\Gamma(\mu+2)}
{}_2F_1\Big(\!\! -n+\mu+1, n+\mu+2;\mu+2;\frac{1+x} 2\Big).
\end{split}
%{}^{\pm}{\mathcal G}_n^{(\mu)}(x):=(1 \mp x)^{\mu+1} \frac{P_{n}^{(\pm(\mu+1), \mp(\mu+1))}(x)}{P_{n}^{(\pm(\mu+1), \mp(\mu+1))}(1)}.
% ;\quad v_{\pm}(x)= (I_{\theta\pm}^{1-s} u^{(m)})(x)
%{}^r{\mathcal G}_n^{(\sigma)}(x):=(1-x^2)^\sigma   \, {}^{r}G_{n-\sigma}^{(\sigma+1/2)}(x).
\end{equation}
From Liu et al. \cite[Definition 2.1 \& (4.30)]{Liu19MC-Optimal}, we find that for   $\lambda \ge 1$ and $\nu\ge 0,$
	\begin{equation}\label{BoundIntLeg-6}
	\max_{|x|\le 1}\Big\{	(1-x^2)^{\lambda-\frac 12}\Big |\, {}_2F_1\Big(\!\!-\nu, \nu+2\lambda;\lambda+\frac 1 2;\frac{1\pm x} 2\Big) \Big| \Big\}\le
	 \frac{\Gamma(\lambda+ 1/2)}{\sqrt \pi}
	\frac{ \Gamma( ({\nu}+1)/ 2)}{\Gamma( ({\nu}+1)/2+\lambda)}.
	\end{equation}
Thus, taking $\nu\to n-\mu-1$ and  $\lambda\to \mu+3/2$ in \eqref{BoundIntLeg-6}, leads to
	\begin{equation}\label{BoundIntLeg-7}
	\max_{|x|\le 1}\Big\{	(1-x^2)^{\mu+1}\Big | {}_2F_1\Big(\!\!-n+\mu+1, n+\mu+2;\mu+2;\frac{1\pm x} 2\Big) \Big| \Big\}\le
	 \frac{\Gamma(\mu+2)}{\sqrt \pi}
	\frac{ \Gamma( (n-\mu)/ 2)}{\Gamma( (n+\mu+3)/2)}.
	\end{equation}
%\frac{P_{n}^{( -\mu-1,\mu+1)}(x)}{P_{n}^{(\mu+1,-\mu-1)}(1)}
Finally the bound \eqref{BoundIntLeg-0} follows from  \eqref{BoundIntLeg-4}, \eqref{BoundIntLeg-5} and \eqref{BoundIntLeg-7}.
\end{proof}

\subsection{$L^\infty$-estimates of Legendre orthogonal projections} With the above preparations, we are now ready to analyse  the  $L^\infty$-error estimate of the $L^2$-orthogonal projection:
\begin{equation}\label{Cbexp}
u(x)=\sum_{n=0}^{\infty}\hat u_n^L\, P_n(x),\quad\;\;   (\pi_N^L  u)(x)= \sum_{n=0}^N \hat u_n^L\, P_n(x).
\end{equation}
%With Theorem \ref{IdentityForUn} at our disposal, we are  now ready to  estimate  the $L^\infty$-errors for the Legendre expansion of functions with AC$_{m}$-BV$_{\!s,\theta}$-regularity.
%
Below, we present the approximation results on the  $L^\infty$-estimate and the weighted $L^\infty$-estimate.  We shall illustrate  that  the former is suboptimal for functions with interior singularity, but optimal for the endpoint singularity, while the latter is optimal in both cases. Such convergence behaviours  were numerically  observed in    \cite{Wang18AML-new,Wang20arXiv-How}, but lack of  theoretical justifications.
\begin{thm}\label{TruncLeg}     Let $u,u',\ldots,u^{(k-1)}\in {\rm AC}([-1,1])$ with $k \in \mathbb{N}.$
\begin{itemize}
\item[(i)] For  $\mu \in (k-1, k)$ and $\theta\in(-1,1),$ if  ${ }^{C}\! D_{\theta+}^{\mu} u \in  {\rm BV}([\theta, 1])$ and
  ${ }^{C}\! D_{\theta-}^{\mu} u \in  {\rm BV}([-1,\theta]),$ then   for  all $N\ge \mu>1/2,$
		\begin{equation}\label{FracLinfA00}
  \begin{split}
  \big\|u-\pi_N^L  u\big\|_{L^\infty (\Omega)}\le \frac{1}{2^{\mu-1} (\mu-1/2)\sqrt{\pi}}\frac{ \Gamma( ({N-\mu+1})/ 2)}{\Gamma( ({N+\mu})/2)} \, U_\theta^{(\mu)},
  \end{split}
  \end{equation}
 and for all $N\ge \mu,$
  \begin{equation}\label{FracLinfB}
  \begin{split}
  \big\|(1-x^2)^{\frac 1 4 }(u-\pi_N^L  u)\big\|_{L^\infty (\Omega)}\le \frac {1 }{2^{\mu-1} \mu \pi}\frac{ \Gamma( ({N-\mu+1})/2)}{\Gamma( ({N+\mu+1})/2)}\, \, U_\theta^{(\mu)},
  \end{split}
  \end{equation}
		where we denoted
\begin{equation}\label{seminormF0}
U_\theta^{(\mu)}:= V_{[-1,\theta]}[{ }^{C}\! D_{\theta-}^{\mu} u]  +V_{[\theta,1]}[{ }^{C}\!D_{\theta+}^{\mu} u]+ |{ }^{C}\! D_{\theta-}^{\mu} u(\theta-)| + |{ }^{C}\! D_{\theta+}^{\mu} u(\theta+)|.
\end{equation}
\item[(ii)]  If $u^{(k)}\in {\rm BV}(\bar \Omega)$ with $\Omega=(-1,1),$ then the estimates  \eqref{FracLinfA00}-\eqref{FracLinfB} with $\mu=k$ hold,  but the total variation $V_{\bar\Omega}[u^{(k)}]$ is in place of $U_\theta^{(\mu)}.$
\end{itemize}
\end{thm}
\begin{proof}
% Using  the identity \eqref{HatUnCaseC},  we obtain from the property \eqref{integralProp} and the bound \eqref{LBoundGeg} (with $\lambda=\sigma+\frac 1 2$) that
Using the identity in Theorem \ref{IdentityForUn} and the   bound in Lemma \ref{BoundIntLeg}, we obtain from \eqref{integralProp}  that %the following bound of the Legendre expansion coefficient:
\begin{equation}\label{HatUnCaseC0}
\begin{split}
|\hat u_n^L| &= \frac{2n+1}{2}
\max_{x\in \bar\Omega} \big\{\big|(I_{1-}^{\mu+1} P_n)(x)\big|, \big|(I_{-1+}^{\mu+1} P_n)(x)\big|\big\}\,
U_\theta^{(\mu)}\\
& \le \frac{(2n+1)\Gamma( ({n-\mu})/ 2)}{ 2^{\mu+2} \sqrt{\pi}\,\Gamma( ({n+\mu}+3)/2)} \,U_\theta^{(\mu)}.
\end{split}
\end{equation}
%where $U_\theta^{\mu}$ is the same as \eqref{seminormF0}.
%
%%We first show that the Legendre expansion coefficient has the bound
%%\begin{equation}\label{BoundForUnRealB}
%%	\begin{split}
%%	|\hat u_n^{L}|\le
%%	\frac{ 2^{\sigma+1} \sqrt{\pi} (2n+1)\Gamma( ({n-\sigma}+1)/ 2)}{\Gamma( ({n+\sigma})/2+1)} \,U_\theta^{m,s}.
%%	\end{split}
%%	\end{equation}
%%Indeed, in Theorem \ref{IdentityForUn}, and the fact   ,
%
%For simplicity, we denote
%$$
%{}^l{\mathcal G}_n^{\sigma}(x):=-(1-x^2)^\sigma   \, {}^{l}G_{n-\sigma}^{(\sigma+1/2)}(x),\quad {}^r{\mathcal G}_n^{\sigma}(x):=(1-x^2)^\sigma   \, {}^{r}G_{n-\sigma}^{(\sigma+1/2)}(x)
%$$
%and we can rewrite  \eqref{HatUnCaseC} as
%
%
%Corollary \ref{sequal1}, Lemma
%\ref{BoundGegB} with  $\lambda=\sigma+1/2 $ and $ \nu=n-\sigma$, we get \eqref{BoundForUnRealB}.

We first prove the error bound \eqref{FracLinfA00}.  For simplicity,  we   denote
	\begin{equation}\label{Snsigma}
	{\mathcal S}_n^\mu:=\frac{ \Gamma( ({n-\mu})/ 2)}{\Gamma( ({n+\mu+1})/2)},\quad
	{\mathcal T}_n^\mu:=\frac{ \Gamma( ({n-\mu})/ 2)}{\Gamma( ({n+\mu}-1)/2)}.
	\end{equation}
Using  the identity $z\Gamma(z)=\Gamma(z+1)$,  we find readily that
	\begin{equation}\begin{split}\label{FracLinfA-1}
	{\mathcal T}_n^\mu-{\mathcal T}_{n+2}^\mu&=
	\frac{{n+\mu}-1} 2\frac{ \Gamma( ({n-\mu})/ 2)}{\Gamma( ({n+\mu+1})/2)}-\frac{{n-\mu}} 2\frac{ \Gamma( ({n-\mu})/ 2)}{\Gamma( ({n+\mu+1})/2)}\\
	&=(\mu-1/2)\frac{ \Gamma( ({n-\mu})/ 2)}{\Gamma( ({n+\mu+1})/2)}= (\mu-1/2)\,{\mathcal S}_n^\mu.
	\end{split}\end{equation}
  As $|P_n(x)|\le 1,$
%Recall the inequality
%\begin{equation}\begin{split}\label{FracLinfA-1-1}
%	\max_{| x|\le 1}|P_n(x)|\le 1,
%	\end{split}\end{equation}
	we derive from  \eqref{HatUnCaseC0} that
	\begin{equation}\begin{split}\label{FracLinfA-2}
	\big|(u -&\pi_N^L  u)(x)\big|\le \sum_{n=N+1}^{\infty}|\hat u_n^{L}| \le
\frac{U_\theta^{(\mu)}}{2^{\mu+2} \sqrt{\pi}}\sum_{n=N+1}^{\infty}
	 \frac{ (2n+1)\Gamma( ({n-\mu})/ 2)}{\Gamma( ({n+\mu+3})/2)}\\
&\le \frac{U_\theta^{(\mu)}}{2^{\mu} \sqrt{\pi}}\sum_{n=N+1}^{\infty}
	 \frac{ \Gamma( ({n-\mu})/ 2)}{\Gamma( ({n+\mu+1})/2)}
= \frac{U_\theta^{(\mu)}}{2^{\mu} (\mu-1/2)\sqrt{\pi}}  \sum_{n=N+1}^{\infty}\big\{{\mathcal T}_n^\mu-{\mathcal T}_{n+2}^\mu\big\}
	%\\&=\frac{U_\theta^{m,s}}{2^{\mu-1} (\mu-3/2)\sqrt{\pi}} \bigg\{{\mathcal T}_{N+1}^\mu+{\mathcal T}_{N+2}^\mu
%	+
%	\sum_{n=N+3}^{\infty}{\mathcal T}_n^\mu-
%	\sum_{n=N+1}^{\infty}{\mathcal T}_{n+2}^\mu\bigg\}
\\&=\frac{U_\theta^{(\mu)}}{2^{(\mu)} (\mu-1/2)\sqrt{\pi}} \big\{{\mathcal T}_{N+1}^\mu+{\mathcal T}_{N+2}^\mu\big\}.
	\end{split}\end{equation}
	Since $\mu>1/2,$ we obtain from \eqref{gammratio} immediately that
	\begin{equation}\label{Tsigmabnd}
	{\mathcal T}_{N+2}^\mu={\mathcal R}_{\mu-1/2}^0(1+(N-\mu)/2) \le {\mathcal R}_{\mu-1/2}^0(1+(N-\mu-1)/2)={\mathcal T}_{N+1}^\mu.
	\end{equation}
Therefore, we have from the above that
\begin{equation*}\begin{split}
	\big|(u -&\pi_N^L  u)(x)\big|\le  \frac{ 2  {\mathcal T}_{N+1}^\mu\, U_\theta^{(\mu)}}{2^{\mu} (\mu-1/2)\sqrt{\pi}}=
	\frac{U_\theta^{(\mu)}}{2^{\mu-1} (\mu-1/2)\sqrt{\pi}}\frac{ \Gamma( ({N-\mu}+1)/ 2)}{\Gamma( ({N+\mu})/2)}.
\end{split}\end{equation*}	
This leads to  the error bound  \eqref{FracLinfA00}.
	%$$\|u-\pi_N^L  u\|_{L^\infty (\Omega)}\le\frac{U_\theta^{m,s}}{2^{\mu+1} \sqrt{\pi}}
%	\frac{ (2n+1)\Gamma( ({n-\mu}+1)/ 2)}{\Gamma( ({n+\mu})/2+1)}$$
\medskip

We now turn to the proof of  \eqref{FracLinfB}. %By Lemma \ref{BoundGegA}, we have
%According to  \cite[(4.20)]{Liu20JAT-Asymptotics} (with $\lambda=1/2$), we have
%		\begin{equation}\label{IneqtyLeg-A}
%	\max_{|x|\le 1}\big\{(1-x^2)^{\frac 14} |P_n(x)|\big\}\le \frac{1}{\sqrt{\pi}} \frac{\Gamma((n+1) / 2)}{\Gamma(n / 2+1)}\le \sqrt{ \frac{2}{ \pi}}\Big(n+\frac{1}{2}\Big)^{-\frac 12},\quad n\ge 0.
%%\frac{1}{\sqrt \pi}	\frac{ \Gamma( (n+1)/ 2)}{\Gamma( n/2+1)}.
%	\end{equation} %\le\frac{1}{\sqrt \pi}(n/2+1/4)^{-1/2 }
Recall the Bernstein inequality (cf.\! \cite{Antonov1981VLUM-An}):
		\begin{equation}\label{IneqtyLeg-A}
	\max_{|x|\le 1}\big\{(1-x^2)^{\frac 14} |P_n(x)|\big\}\le \sqrt{ \frac{2}{ \pi}}\Big(n+\frac{1}{2}\Big)^{-\frac 12},\quad n\ge 0.
%\frac{1}{\sqrt \pi}	\frac{ \Gamma( (n+1)/ 2)}{\Gamma( n/2+1)}.
	\end{equation} %\le\frac{1}{\sqrt \pi}(n/2+1/4)^{-1/2 }
Thus we infer from \eqref{HatUnCaseC0}  and \eqref{IneqtyLeg-A} that
		\begin{equation}\begin{split}\label{FracLinfB-1}
e_N(x)& := \big|(1-x^2)^{\frac 14}  (u -\pi_N^L  u)(x)\big| \le \sum_{n=N+1}^{\infty}  \max_{|x|\le 1}\big\{(1-x^2)^{\frac 14} |P_n(x)|\big\}\, |\hat u_n^L |
	 \\
&\le
\frac{U_\theta^{(\mu)}}{2^{\mu} \pi}\sum_{n=N+1}^{\infty}
	\sqrt{\frac n2+\frac 1 4}\, \frac{ \Gamma( ({n-\mu})/ 2)}{\Gamma( ({n+\mu+3})/2)}.
%\frac{U_\theta^{m,s}}{2^{\mu} \pi}\sum_{n=N+1}^{\infty}
%	 \frac{(n/2+1) \Gamma( (n+1)/ 2)}{\Gamma( n/2+1)}\frac{ \Gamma( ({n-\mu}+1)/ 2)}{\Gamma( ({n+\mu})/2+1)}.
	\end{split}\end{equation}
Considering   \eqref{IneqtyLeg-1} with $z=z_n=(n+\mu+1)/2$ and  $c=1/4-(\mu+1)/2\, (\le 1/4)$, we find  from  its monotonicity    that
$\widehat {\mathcal R}_c(z_n)\ge  \widehat {\mathcal R}_c(\infty)=1$   (cf. \eqref{NRatioGammaV2-1}).  This immediately implies
	\begin{equation}\begin{split}\label{FracLinfB-2}
\frac{\sqrt{n/2+1/4}}{\Gamma( ({n+\mu+3})/2)}&\le
	 \frac{1}{\Gamma( ({n+\mu})/2+1)},
	\end{split}\end{equation}
so we can bound the summation in 	\eqref{FracLinfB-1} by
		\begin{equation}\begin{split}\label{FracLinfB-3}
	\sum_{n=N+1}^{\infty}
	 \sqrt{\frac n2+\frac 14}\,\frac{ \Gamma( ({n-\mu})/ 2)}{\Gamma( ({n+\mu+3})/2+1)}&\le\sum_{n=N+1}^{\infty}
	 \frac{ \Gamma( ({n-\mu})/ 2)}{\Gamma( ({n+\mu})/2+1)}.
	\end{split}\end{equation}
	Similarly,  denoting
\begin{equation}\label{Snsigma-2}
\widehat{\mathcal S}_n^\mu:=\frac{ \Gamma( ({n-\mu})/ 2)}{\Gamma( ({n+\mu})/2+1)},\quad
\widehat{\mathcal T}_n^\mu:=\frac{ \Gamma( ({n-\mu})/ 2)}{\Gamma( ({n+\mu})/2)} ,
\end{equation}
we  find readily that
\begin{equation}\begin{split}\label{FracLinfB-4}
\widehat{\mathcal T}_n^\mu-\widehat{\mathcal T}_{n+2}^\mu& =
\frac{{n+\mu}} 2\frac{ \Gamma( ({n-\mu})/ 2)}{\Gamma( ({n+\mu})/2+1)}-\frac{{n-\mu}} 2\frac{ \Gamma( ({n-\mu})/ 2)}{\Gamma( ({n+\mu})/2+1)}
%\\&=(\mu-1)\frac{ \Gamma( ({n-\mu}+1)/ 2)}{\Gamma( ({n+\mu}+1)/2)}
= \mu\,\widehat{\mathcal S}_n^\mu.
\end{split}\end{equation}
Following the same lines as in the derivation of \eqref{FracLinfA-2},  we can get %find from the above that %By \eqref{FracLinfB-1} and \eqref{FracLinfB-3}, we get
\begin{equation}\begin{split}\label{FracLinfB-5}
e_N(x) & \le
%\frac{U_\theta^{m,s}}{2^{\mu-1} \pi}\sum_{n=N+1}^{\infty}
%\widehat{\mathcal S}_n^\mu
%= \frac{U_\theta^{m,s}}{2^{\mu-1}  (\mu-1)\pi}  \sum_{n=N+1}^{\infty}\big\{\widehat{\mathcal T}_n^\mu-\widehat{\mathcal T}_{n+2}^\mu\big\}
%\\&=\frac{U_\theta^{m,s}}{2^{\mu-1}  (\mu-1)\pi} \bigg\{\widehat{\mathcal T}_{N+1}^\mu+\widehat{\mathcal T}_{N+2}^\mu
%+
%\sum_{n=N+3}^{\infty}\widehat{\mathcal T}_n^\mu-
%\sum_{n=N+1}^{\infty}\widehat{\mathcal T}_{n+2}^\mu\bigg\}\\&=
\frac{U_\theta^{(\mu)}}{2^{\mu}  \mu\pi} \big\{\widehat{\mathcal T}_{N+1}^\mu+\widehat{
\mathcal T}_{N+2}^\mu\big\}
\le \frac{ 2 \widehat{\mathcal T}_{N+1}^\mu}{2^{\mu}  (\mu-1)\pi}  \, U_\theta^{(\mu)},
	\end{split}\end{equation}
where we used the property derived from \eqref{gammratio} with $\mu>1,$ that is,
	\begin{equation*}\label{Tsigmabnd-2}
\widehat{\mathcal T}_{N+2}^\mu={\mathcal R}_{\mu-1}^0(1+(N-\mu)/2) \le {\mathcal R}_{\mu-1}^0(1+(N-\mu-1)/2)=\widehat{\mathcal T}_{N+1}^\mu.
	\end{equation*}
Then  the estimate \eqref{FracLinfB}  follows from \eqref{Snsigma-2} and \eqref{FracLinfB-5} straightforwardly.

\smallskip
For $\mu=k,$  using the identity in Corollary \ref{sequal1} to derive the bound in \eqref{HatUnCaseC0}, and then following the same lines, we can obtain the estimates in (ii).
\end{proof}

 Under the regularity assumption in Theorem \ref{TruncLeg},  we infer from \eqref{NRatioGammaV2-1} and the estimates \eqref{FracLinfA00}-\eqref{FracLinfB} the convergence behaviour:
\begin{equation}\label{converorg}
\|u-\pi_N^L  u\|_{L^\infty (\Omega)}=O(N^{-\mu+1/2}),\quad  \|(1-x^2)^{\frac 1 4 }(u-\pi_N^L  u)\|_{L^\infty (\Omega)}=O(N^{-\mu}),
\end{equation}
which exhibit  a half-order convergence difference. Moreover,  the estimate \eqref{FracLinfB-1} implies
%for any $x\in [-a,a]\subset (-1,1),$ we have the point-wise estimate
\begin{equation}\label{converorg2}
 \big| (u -\pi_N^L  u)(x)\big|= (1-x^2)^{-\frac 14}\, O(N^{-\mu}),\quad\forall\, x\in [-a,a]\subset (-1,1).
\end{equation}
For a  function with an interior singularity  in $[-a,a]$ with $|a|<1$,   one expects the optimal order  $O(N^{-\mu}).$
Note  from  \eqref{FracLinfA-2} and \eqref{FracLinfB-1} that the bounds essentially depend on the maximum of $|P_n(x)|$
and $(1-x^2)^{1/4}|P_n(x)|,$ which behave very differently near the endpoints  as
shown in Figure \ref{PnWithWeight}.  In fact, $|P_n(x)|=O(n^{-1/2})$ for $x\in [-a,a],$ but it is overestimated by the bound $1$ at $x=\pm 1.$ However, from \eqref{IneqtyLeg-A}, we have $(1-x^2)^{1/4}|P_n(x)|\le Cn^{-1/2}$ for all $x\in [-1,1].$ This is actually the cause of  the lost order  in  the (non-weighted) $L^\infty$-estimate in \eqref{converorg}.
\begin{figure}[!ht]
	\begin{center}
		{~}\hspace*{-20pt}	
\includegraphics[width=0.42\textwidth]{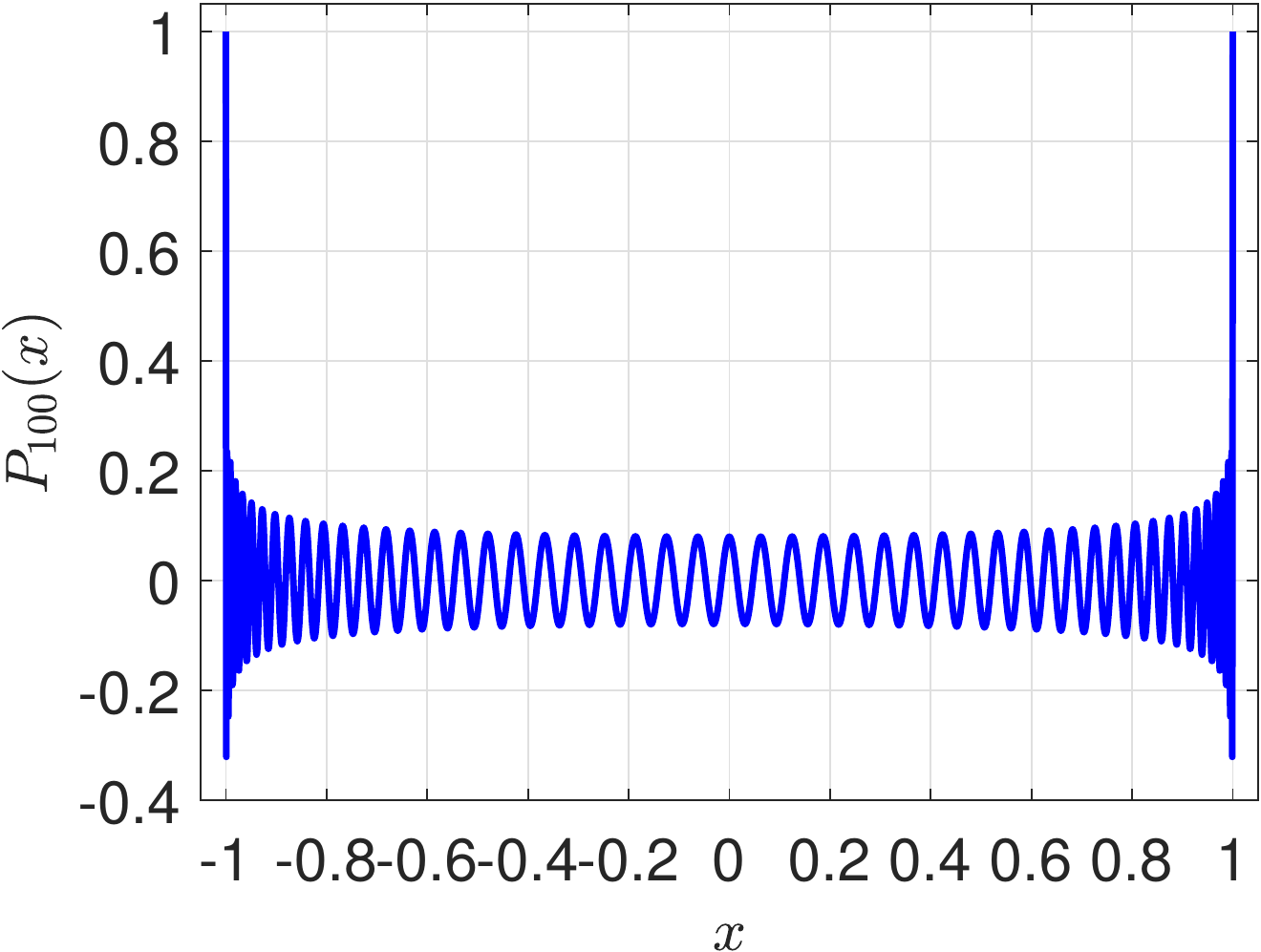}\qquad\quad
		\includegraphics[width=0.42\textwidth]{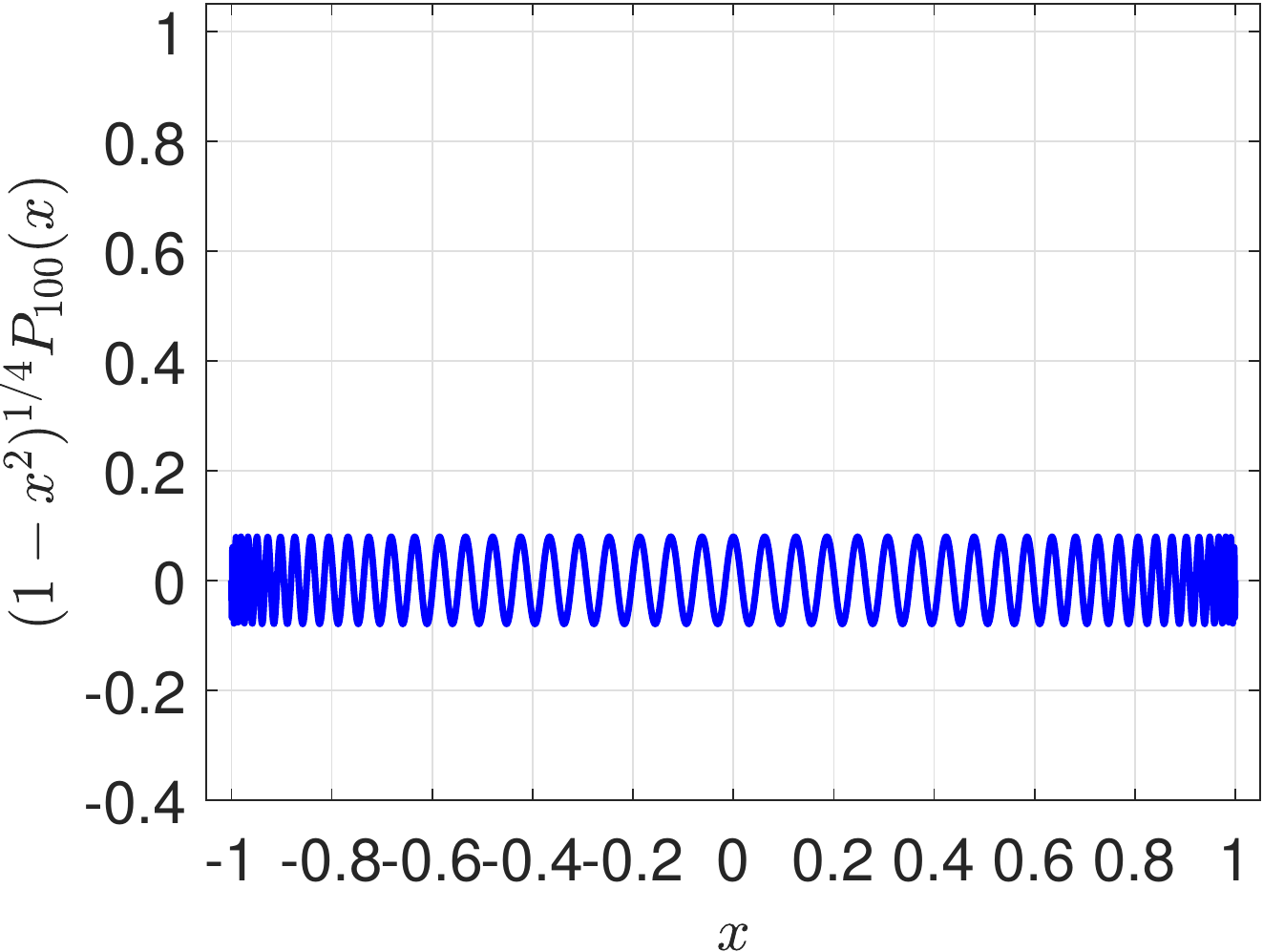}
		\caption{$P_n(x)$ and $(1-x^2)^{\frac 14}P_n(x)$ with $x\in [-1,1]$ and $n=100.$}
		\label{PnWithWeight}
	\end{center}
\end{figure}

With these analysis tools  at hand,  we  further examine $u(x)=|x|$ (as a motivative example in Trefethen  \cite{Trefethen2008SIREV}), for which
 Wang \cite{Wang18AML-new,Wang20arXiv-How} observed the order $O(N^{-1})$ numerically,  but  the order is $O(N^{-1/2})$ based on the error estimate of Legendre approximation in $L^\infty$-norm.    From the pointwise error plots in Figure \ref{PoWisePosi} (left),  we see the largest error occurs at the singular point $x=0.$ Indeed, we have the following estimates (with the proof given in Appendix  \ref{AppendixC}), which are sharp as shown in Figure \ref{PoWisePosi} (right).  %The  optimal estimate is still missing.
  \begin{theorem}\label{01estimate} Consider $u(x)=|x|$ for  $x\in [-1,1].$  Then for $N>2,$ we have
 \begin{equation}\label{un|x|-1+1}
\begin{split}
\big|(u -\pi_N^L  u)(0)\big|\le \frac{2}{\pi(N-1)}\,; \; \quad \big|(u-\pi_N^L  u)(\pm 1)\big|&\le
\frac{1 }{2\sqrt{\pi}}\frac{\Gamma( N/2-1)} {\Gamma( N/2+1/2)}.
\end{split}
\end{equation}
%and
%\begin{equation}\label{un|x|-20}
%\begin{split}
%\big|(u-\pi_N^L  u)(\pm 1)\big|&\le
%\frac{1 }{2\sqrt{\pi}}\frac{\Gamma( N/2-1)} {\Gamma( N/2+1/2)}.
%\end{split}
%\end{equation}
 \end{theorem}
 \begin{figure}[!ht]
	\begin{center}
		{~}\hspace*{-20pt}	
\includegraphics[width=0.42\textwidth]{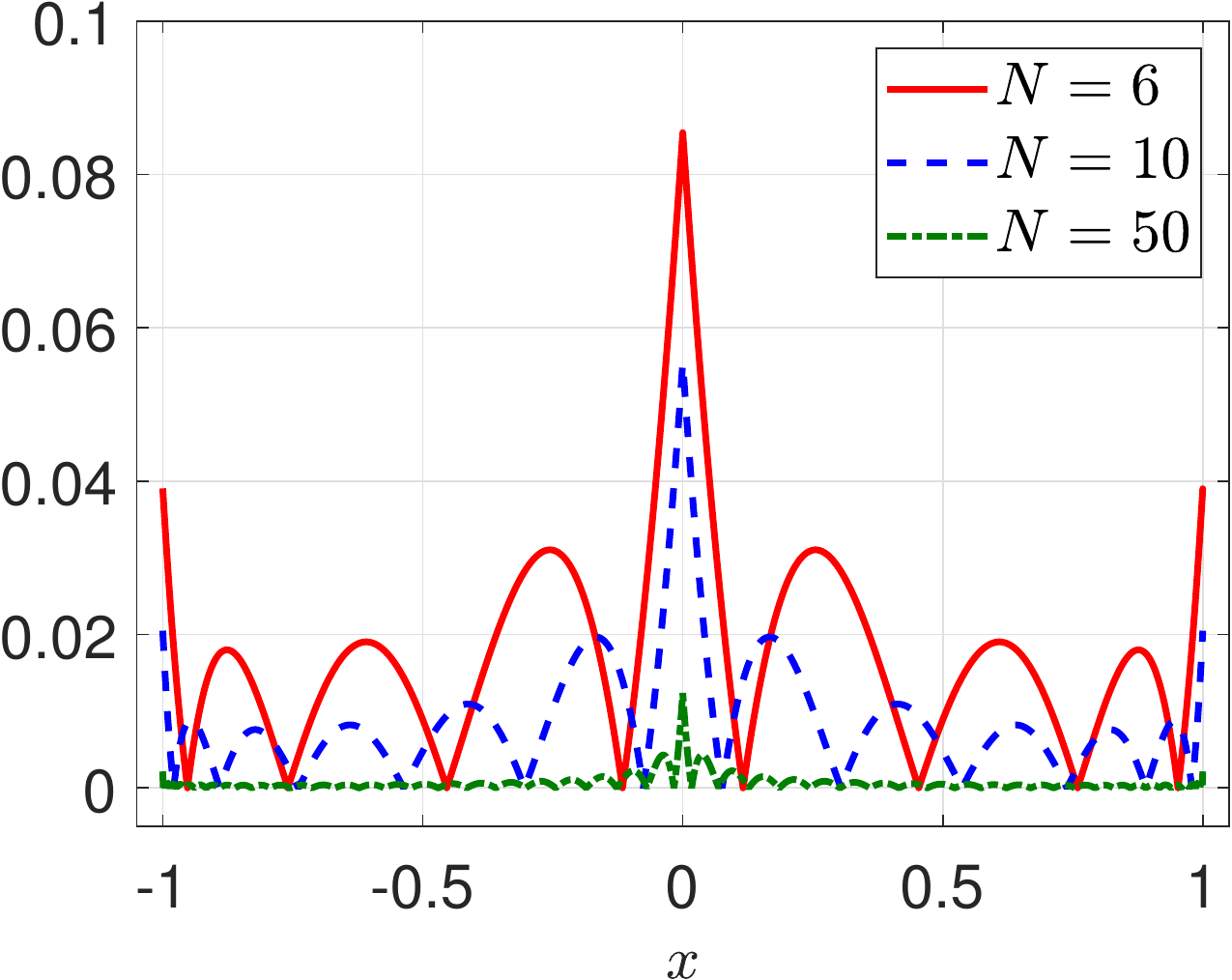}\qquad\quad
		\includegraphics[width=0.42\textwidth]{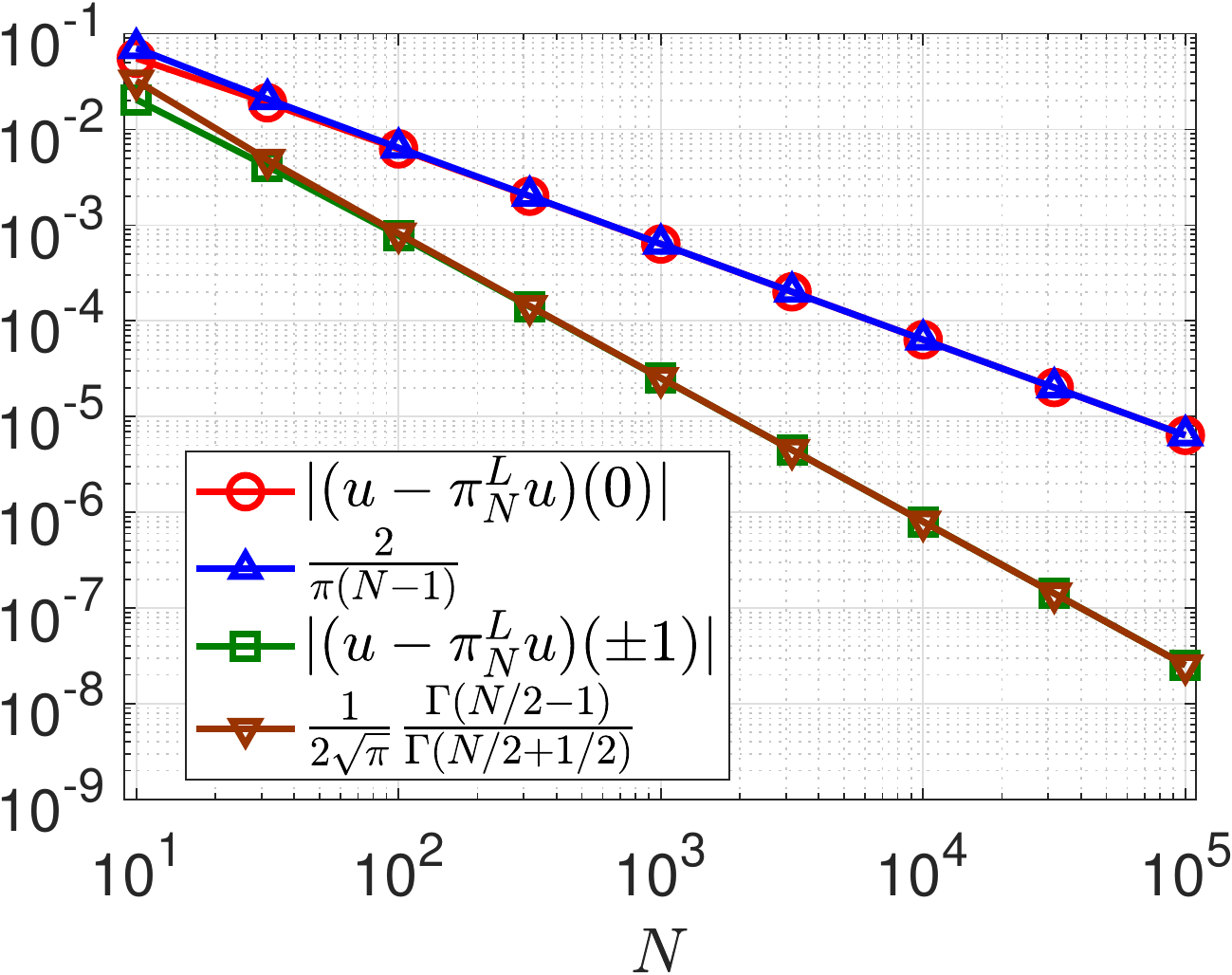}%{WeightPoWisePosi}
		\caption{Left: $|(u -\pi_N^L  u)(x)|$  with $u(x)=|x|$ and different $N.$ Right: Tightness of the bounds in Theorem \ref{01estimate}. }
		\label{PoWisePosi}
	\end{center}
\end{figure}

%{\color{red}\bf  Need to describe more about this example about the regularity!}
Finally, we apply the main results to the example $u(x)=|x|^\mu$ with $\mu \in(k-1,k).$  As shown in \cite[Thm. 4.3]{Liu19MC-Optimal},
 $u,u',\ldots,u^{(k-1)}\in {\rm AC}(\bar\Omega),$ ${ }^{C}\! D_{0+}^{\mu} u \in  {\rm BV}([0, 1]),$
 and ${ }^{C}\! D_{0-}^{\mu} u \in  {\rm BV}([-1,0]).$ Thus, we infer from \eqref{converorg} that the expected convergence orders are
 $O(N^{-\mu+1/2})$ in $L^\infty$-norm and  $O(N^{-\mu})$ in $L^\infty_{\varpi}$-norm with $\varpi=(1-x^2)^{1/4}.$ Observe from
 the numerical results in Table \ref{Tab-1}  that the latter is optimal, but the former loses half order.
%  From Theorem \ref{TruncLeg} and   \eqref{RateGamma-0-4}, we obtain the \eqref{converorg} hold with  $\mu$ is in place of $\alpha.$
%
%%For $u=|x|\in {\rm AC}(\Omega),~u'={\rm sign} (x)\in{\rm BV}(\Omega), $
%%where ${\rm sign} (x)$ is the sign functions.
%
%Table \ref{Tab-1} shows the optimal
%convergence orders for $L^\infty_\varpi$-norm, the sub-optimal %and $L^2$-norm
%convergence orders  for $L^\infty$-norm. In this case show the errors of $L^\infty$-norm and $L^\infty_\varpi$-norm are same, which implies that
%the maximum errors is  obtained when $x$ close to 0. In our proof for errors of $L^\infty$-norm lost order in $x$ near 1.
\begin{table}[!htbp]
 \centering
 \caption{Convergence order of Legendre expansion for $|x|^\mu$.} %of  $\|u-\pi_N^C u\|_{L^\infty (\Omega)}$ and $\|u-\pi_N^C u\|_{L^2_\omega(\Omega)}$.}
\small
 \begin{tabular}{|c|c|c|c|c|c|c|c|c|}
  \hline
  \multicolumn{1}{ |c  }{\multirow{2}{*}{$N$} } &
  \multicolumn{4}{ |c| }{Errors in $L^\infty$-norm} & \multicolumn{4}{ |c| }{Errors in $L^\infty_\varpi$-norm}      \\ \cline{2-9}
  \multicolumn{1}{ |c  }{}                        &
  \multicolumn{1}{ |c| }{$\mu=1.7$}  & order  & $\mu=2.6$  & order  &$ \mu=1.7$  & order & $\mu=2.6$  & order \\ \hline
$ 2^3 $ & 5.81e-03 &--& 2.35e-03 &-- &5.81e-03 &--& 2.22e-03 &--  \\
$ 2^4 $ & 2.03e-03 & 1.52 & 4.38e-04 & 2.42 & 2.03e-03 & 1.52 & 4.38e-04 & 2.34 \\
$ 2^5 $ & 6.72e-04 & 1.60 & 8.01e-05 & 2.45  & 6.72e-04 & 1.60 & 8.01e-05 & 2.45 \\
$ 2^6 $ & 2.15e-04 & 1.65 & 1.40e-05 & 2.52 & 2.15e-04 & 1.65 & 1.40e-05 & 2.52  \\
$ 2^7 $ & 6.74e-05 & 1.67 & 2.37e-06 & 2.56  & 6.74e-05 & 1.67 & 2.37e-06 & 2.56\\
$ 2^8 $ & 2.09e-05 & 1.69 & 3.97e-07 & 2.58 & 2.09e-05 & 1.69 & 3.97e-07 & 2.58\\
  \hline
 \end{tabular} \label{Tab-1}
\end{table}
\subsection{$L^2$-estimates}
As pointed out in \cite[Chap.\!\! 3]{Shen2011Book}, the estimate of the $L^2$-orthogonal projection is the starting point  to derive many other approximation results that  provide fundamental tools for error analysis of spectral and $hp$ methods  (see, e.g., \cite{Bernardi1997Book,Canuto2006Book,Guo1998Book,Hesthaven2007Book,Schwab1998Book,Shen2011Book}).  Most    estimates  therein are for functions in Sobolev or Besov spaces. Here, we consider functions with
 AC-BV regularity, thereby enriching the approximation theory.
%
%
%e.g.,  $H^1_{0}$-orthogonal projection  and Legendre-Gauss quadrature/interpolation.  In light of Theorem {\rm \ref{TruncLeg}} and Theorem {\rm \ref{newL2est}},     a  set of new approximation results can be obtained for the functions with the regularity considered herein.
%
%
%
%The AC-BV regularity  for
%polynomial approximation  not  consider  in  some books (see, e.g.,

 %For example,   we consider the Legendre-Gauss-Lobatto interpolation {\bf add some here!}

We first highlight the fundamental importance of estimating $L^2$-orthogonal projection in \eqref{Cbexp}.
For $u\in H^1(\Omega),$ we define
\begin{equation}\label{H1proj}
(\pi_{N}^1 u)(x)= u(-1)+\int^x_{-1}\pi_{N-1}^L u'(t)\, \rd t\in {\mathcal P}_{\!N},
\end{equation}
where ${\mathcal P}_{\!N}$ denotes the set of  polynomials of degree at most $N.$ Note that for $N\ge 2,$ we have from the orthogonality of Legendre polynomials that
$$(\pi_{N}^1 u)(1)= u(-1)+\int^1_{-1}\pi_{N-1}^L u'(t)\,\rd t=  u(-1)+ \int^1_{-1} u'(t)\,\rd t=u(1).$$
Thus we have $(\pi_{N}^1 u)(\pm 1)=u(\pm 1).$ Moreover, one verifies readily that
$$\int^1_{-1}\big( \pi_{N}^1 u-u\big)'(x) \,v'(x)\,\rd x =0, \;\;\; \forall v\in \mathcal{P}_{\!N}^0:=\{v\in {\mathcal P}_{\!N} : v(\pm 1)=0 \}.$$
Therefore, \eqref{H1proj} defines the $H^{1}_0$-orthogonal projection. Note that
\begin{equation}\label{H1proj2}
\|(u-\pi_{N}^1 u)'\|_{L^{2}(\Omega)}^{2}=\|u'-\pi_{N-1}^Lu' \|_{L^{2}({\Omega})}^{2},
\end{equation}
so the $H^1$-estimate boils down to the estimate of the $L^2$-orthogonal projection. The high-order $H_0^m$-orthogonal projection is treated similarly in a recursive manner (see, e.g.,  \cite{Bernardi1997Book}). On the other hand, the analysis of Gauss-type interpolation and quadrature errors is also based upon the Legendre expansion (see \cite{Shen2011Book}).

With tools in Subsection \ref{sub31}, we can also derive  the following optimal  $L^2$-error bound under  the  AC-BV regularity of $u$, from which we can further establish many other approximation results indispensable for analysis of spectral and $hp$ methods for PDEs. Here, we omit such extensions.
\begin{thm}\label{newL2est} Assume the conditions in  Theorem {\rm \ref{TruncLeg}} hold.
\begin{itemize}
\item[(i)]
 For $\mu\in(k-1,k)$ and $ -1/2<\mu< N,$
		\begin{equation}\label{FracL2B}
		\begin{split}
		\|u-\pi_N^L  u\|_{L^2(\Omega)}\le \sqrt{\frac{2}{(2\mu+1) \pi}
		\frac{ \Gamma( {N-\mu})}{\Gamma( {N+\mu+1})}} \, U_\theta^{(\mu)}.
		\end{split}
		\end{equation}	
	%{\color{red}	If $\mu=k,$ we have ???}
\item[(ii)] For $\mu=k,$
 the estimates  \eqref{FracL2B} hold,  with the total variation $V_{\bar\Omega}[u^{(k)}]$ is in place of $U_\theta^{(\mu)}.$
\end{itemize}
\end{thm}
\begin{proof} Similar to   \eqref{Tsigmabnd}, we can use \eqref{gammratio} to show that
$$\frac{ \Gamma( ({n-\mu})/ 2)}{\Gamma( ({n+\mu+3})/2)}\le
	\frac{ \Gamma( ({n-\mu-1})/ 2)}{\Gamma( ({n+\mu})/2+1)}.$$
Then from \eqref{DuplicationFormula},	we derive
	\begin{equation}\begin{split}\label{FracL2-4}
	& \frac{ (n/2+1/4)\Gamma^2( ({n-\mu})/ 2)}{\Gamma^2( ({n+\mu+3})/2)} \le \frac{ \Gamma^2( ({n-\mu})/ 2)}{\Gamma( ({n+\mu+3})/2)\Gamma( ({n+\mu+1})/2)}\\
	&\qquad\qquad  \le \frac{ \Gamma( ({n-\mu})/ 2)\Gamma( ({n-\mu-1})/ 2)}{\Gamma( ({n+\mu})/2+1)\Gamma( ({n+\mu+1})/2)}=2^{2(\mu+1)}\frac{ \Gamma( {n-\mu-1})}{\Gamma( {n+\mu+1})}\\
&\qquad\qquad  =\frac{2^{2(\mu+1)}}{2\mu+1}
\bigg(\frac{ \Gamma( {n-\mu-1})}{\Gamma( {n+\mu})}-\frac{ \Gamma( {n-\mu})}{\Gamma( {n+\mu+1})}\bigg).
	\end{split}\end{equation}
%Recall that
%\begin{equation}\begin{split}\label{FracL2-4-1}
%\int^1_{-1}P_m(x)P_n(x)dx =\delta_{mn}\frac{2}{2n+1},
%\end{split}\end{equation}
%where $\delta_{mn}$ is the Kronecker Delta symbol.
Then, by the orthogonality of Legendre polynomials, we derive from
\eqref{HatUnCaseC0} and  \eqref{FracL2-4} that for $\mu> -1/2$,
	\begin{equation}\label{FracL2-5}
	\begin{split}
	\big\|u-\pi_N^{L} u\big\|_{L^2(\Omega)}^2=&\sum_{n=N+1}^{\infty}\frac{2}{2n+1}\big|\hat u_n^{L}\big|^2
	\le   \frac{(U_\theta^{(\mu)})^2}{2^{2\mu+3} \pi}\sum_{n=N+1}^{\infty}
	 \frac{ (2n+1)\Gamma^2( ({n-\mu})/ 2)}{\Gamma^2( ({n+\mu+3})/2)}\\
\le &   \frac{2(U_\theta^{(\mu)})^2}{(2\mu+1) \pi}\frac{ \Gamma( {N-\mu})}{\Gamma( {N+\mu+1})}.
	\end{split}
	\end{equation}

For $\mu=k,$  using the identity in Corollary \ref{sequal1} to derive the bound in \eqref{HatUnCaseC0}, and then following the same lines, we can obtain the estimates in (ii).
		\end{proof}

\section{Legendre expansion of functions with endpoint singularities}\label{sect4main}
\setcounter{equation}{0}
\setcounter{lmm}{0}
\setcounter{thm}{0}
\setcounter{cor}{0}

%$\mu \in (k-1, k]$ with $k \in \mathbb{N}$ and let $\theta\in(-1,1).$  If  $u,u',\ldots,u^{(k-1)}\in {\rm AC}([-1,1]),$  ${ }^{C}\! D_{\theta+}^{\mu} u \in  {\rm BV}([\theta, 1])$

The aforementioned AC-BV framework and  main results can be extended to the study of  the end-point singularities, which typically occur in underlying solutions  of  PDEs  in various situations, for instance, irregular domains, singular coefficients and mismatch of boundary conditions among others. It is known that  the  Legendre expansion of a function with an endpoint singularity  has  a much higher convergence rate than that with an interior singularity of the same type.  We illustrate this through an example which also motivates the seemingly complicated extension. To fix the idea, we focus on the left endpoint singularity but the results can be extended to the right endpoint setting straightforwardly.
\subsection{An illustrative example} We consider $u(x)=(1+x)^\mu g(x)$ with $\mu\in (k-1,k), k\in\mathbb N$ and
 and a sufficiently smooth $g(x)$ on $\Omega.$ Then we can  write
\begin{equation}\label{newA}
u(x)=(1+x)^\mu g(x)=\sum_{m=0}^\infty\frac{g^{(m)}(-1)}{m!} (1+x)^{\mu+m}.
%and set $s$ to be the fractional part of $\alpha,$ i.e.,  $\alpha=k-1+s=\sigma-1.$
\end{equation}
Then by  \eqref{intformu},
\begin{equation}\label{Dcregul}
\begin{split}
({ }^{C}\! D_{-1+}^{\mu} u)(x)=\big(I_{-1+}^{k-\mu} u^{(k)}\big)(x) & = \sum_{m=0}^\infty\frac{\{\mu+m\}_{k}}{m!} g^{(m)}(-1) \,
I_{-1+}^{k-\mu}\big\{ (1+x)^{m+\mu-k}\big\}\\
& = \sum_{m=0}^\infty\frac{\{\mu+m\}_{k}\, \Gamma(m+\mu-k+1)}{(m!)^2} g^{(m)}(-1) \,
 (1+x)^{m},
\end{split}
\end{equation}
where  $\{a\}_k=a(a-1)\cdots (a-k+1)$ stands for the falling factorial. This implies  $({ }^{C}\! D_{-1+}^{\mu} u)(x)$ is sufficiently smooth. In particular, if $g=1,$ then $({ }^{C}\! D_{-1+}^{\mu} u)(x)$ is equal to a constant.

%To fix the idea, we focus on the left end-point singularity at  $x=-1$.
We deduce from  Theorem  \ref{IdentityForUn} with $\theta\to-1^+$ that  for   $u, u',\cdots, u^{(k-1)}\in {\rm AC}(\bar \Omega),$ and
$  { }^{C}\! D_{-1+}^{\mu}u  \in  {\rm BV}(\bar\Omega)$ with $\mu \in (k-1, k],$ we have
\begin{equation}\label{HatUnCaseC-3Old-1}
\begin{split}
\hat u_n^L = \frac{2n+1}{2} \Big\{
  ({{ }^{C}\! D_{-1+}^{\mu} u)(-1+)} (I_{1-}^{\mu+1} P_n)(-1)
+\int_{-1}^{1} (I_{1-}^{\mu+1} P_n)(x) \,  \rd \big\{{ }^{C}\! D_{-1+}^{\mu} u(x)\big\}\Big\}.
%\hat u_n^L= &\frac{2n+1}{2} \bigg\{ %(-1)^{ n+[ n-s]}
%{}^{+}{\mathcal G}_n^{(\mu)}(\theta)\, { }^{C}\! D_{\theta+}^{\mu} u(\theta+)+ \int_\theta^1    {}^+{\mathcal G}_n^{(\mu)}(x)  \, {\rd\,\big\{ { }^{C}\! D_{\theta+}^{\mu} u(x)}\big\}\\&
%  \quad
%\!\!+\,{}^-{\mathcal G}_n^{(\mu)}(\theta)\, { }^{C}\! D_{\theta-}^{\mu} u(\theta-) -\int_{-1}^\theta    {}^-{\mathcal G}_n^{(\mu)}(x)  \, {\rd\, \big\{ { }^{C}\! D_{\theta-}^{\mu} u(x)}\big\}\bigg\},
\end{split}
\end{equation}
%where $\sigma=m+s,$ $v(x):= (I_{-1+}^{1-s} u^{(m)})(x),$ and ${}^r{\mathcal G}_n^{(\sigma)}(x) $ is defined in  \eqref{twoGGFs}.
In fact, we can  show   that $(I_{1-}^{\mu+1} P_n)(-1)\sim n^{-2(\mu+1)},$ so the first term decays like $O(n^{-2\mu-1})$ (which gives  the optimal convergence order for \eqref{newA} (see Table \ref{UnV2}),  and  doubles $O(n^{-\mu-1/2})$ for the interior singularity, e.g., of $|x|^\mu g(x)$).
%as ${ }^{C}\! D_{-1+}^{\mu} u(x)$ is sufficiently smooth.
%upon the formula
%  \eqref{HatUnCaseC-3Old-1}, as  the Caputo fractional derivatives ${ }^{C}\! D_{-1+}^{\mu} u(x)$ adds more regularity to $u^{(k)}$.
\begin{lemma}\label{Pnfyi} For $n\ge \mu+1>0,$ we have
\begin{equation}\label{HatUnCaseC-3Old-1+1}
\begin{split}
(I_{1-}^{\mu+1} P_n)(-1)&=  \frac{(-1)^n2^{\mu+1}\Gamma(\mu+1)\sin((\mu+1)\pi)}{\pi} \frac{\Gamma(n-\mu)}{\Gamma(n+\mu+2)}.
\end{split}
\end{equation}
\end{lemma}
\begin{proof}  By \eqref{twoGGFs}, we have
\begin{equation}\label{twoGGFs0}
\begin{split}
&(I_{1-}^{\mu+1} P_n)(-1)=
\frac{2^{\mu+1}}{\Gamma(\mu+2)} \frac{P_{n}^{(\mu+1, -\mu-1)}(-1)}{P_{n}^{(\mu+1,-\mu-1)}(1)}
=
\frac{(-1)^n 2^{\mu+1}}{\Gamma(\mu+2)} \frac{P_{n}^{(-\mu-1, \mu+1)}(1)}{P_{n}^{(\mu+1,-\mu-1)}(1)},
\end{split}
\end{equation}
where we used the parity $P_n^{(\alpha,\beta)}(-x)=(-1)^n P_{n}^{(\beta,\alpha)}(x)$ valid for all real parameters $\alpha,\beta$
(cf.\!   Szeg\"o  \cite[p.\! 64]{Szego1975Book}).  Then we derive \eqref{HatUnCaseC-3Old-1+1} from \eqref{Pnmumu} and \eqref{nonitA} immediately.
\end{proof}
%
%
%Recall the properties (\cite[Ch.2]{Andrews1999Book})
%\begin{equation}\label{Nist15421cc}
%\lim_{z\to 1^-}\frac{ {}_2F_1(a,b;c;z)}{(1-z)^{c-a-b}}=\frac{\Gamma(c)\Gamma(a+b-c)}{\Gamma(a)\Gamma(b)},\quad {\rm if}\;\; c<a+b;
%\end{equation}
%and
%\begin{equation}\label{nonitA}
%\Gamma(1-a)\Gamma(a)=\frac{\pi} {\sin (\pi a)}.
%\end{equation}
%Together with \eqref{BoundIntLeg-4}, we obtain

We find from Lemma \ref{BoundIntLeg} that the second term in \eqref{HatUnCaseC-3Old-1} decays at a rate
$O(n^{-\mu-1/2}),$ if one naively works this out with this formula.  However, in view of \eqref{Dcregul},  we can continue to carry out integration by parts upon \eqref{HatUnCaseC-3Old-1} as many as times we want, until the first  boundary term in \eqref{HatUnCaseC-3Old-1} dominates the error.
 This produces the optimal order $O(n^{-2\mu-1})$ (see Table \ref{UnV2} for numerical illustrations).
\begin{table}[!htbp]
	\centering
	\caption{Decay rate of $|\hat u_n^{L}|$ with  $u= (x+1)^\mu \sin x$.}
\small
	\begin{tabular}{|c|c|c|c|c|c|c|}
		\hline
	$n$&{$\mu=0.1$}  & order
		&$ \mu=1.2$  & order  & $\mu=2.6$  & order
		\\
\hline
$2^3$&1.26e-02&--& 6.86e-04&--& 1.42e-04&--\\
$2^4$&5.59e-03&1.17& 6.59e-05&3.38& 1.35e-06&6.72\\
$2^5$&2.47e-03&1.18& 6.41e-06&3.36& 1.86e-08&6.18\\
$2^6$&1.08e-03&1.19& 6.20e-07&3.37& 2.60e-10&6.16\\
$2^7$&4.73e-04&1.19& 5.94e-08&3.38& 3.49e-12&6.22\\
\hline		
	\end{tabular} \label{UnV2}
\end{table}

\subsection{Approximation results for functions with endpoint singularities}
With the above understanding, we are now ready to present the identity on
 the Legendre expansion coefficient from \eqref{HatUnCaseC-3Old-1} and integration by parts. Given that  the function has more regularity in this case,  we make the following assumption.
\begin{defn}[{\bf Regularity Assumption}]\label{functionA-1} For $\mu \in (k-1, k]$ with $k\in \mathbb N,$  assume
 $u, \cdots, u^{(k-1)}\in {\rm AC}(\bar \Omega)$ and
$ v_\mu(x):={ }^{C}\! D_{-1+}^{\mu}u  \in  {\rm BV}(\bar\Omega).$  We further assume that $v_\mu, \cdots,v^{(m-1)}_\mu\in  {\rm AC}(\bar \Omega)$ and  $v^{(m)}_\mu\in {\rm BV}(\bar \Omega).$   Accordingly, we denote
 \begin{equation}\label{seminormF0-1}
  \begin{split}
  U^{(\mu,m)}_-:= V_{\bar\Omega}[v^{(m)}_\mu]  + \,|\sin(\mu\pi)| \sum_{l=0}^{m}\big| v^{(l)}_\mu (-1+) \big|.
%  \int_{-1}^\theta   \big|\rd \big\{I_{\theta-}^{1-s}u^{(m)} \big\} (x)\big|
%  +\int^{1}_\theta    \big|\rd \big\{I_{\theta+}^{1-s}u^{(m)} \big\} (x)\big|\\
%  &+\big|\big\{  I_{\theta-}^{1-s} u^{(m)}\big\}(\theta+)\big|+\big|\big\{  I_{\theta+}^{1-s} u^{(m)}\big\}(\theta-)\big|.
  \end{split}
  \end{equation}
For simplicity, {\em we say $u$ is  of  {\rm AC-BV$_{\!\!\mu,m}$-regularity}.} \qed
\end{defn}
%for    some  $\mu\in(k-1,k]$ with $k \in \mathbb{N},$ and $m\in {\mathbb N}_0$, if {\rm (i)} $ u, u',\cdots, u^{(k-1)}\in {\rm AC}(\bar \Omega);$  {\rm (ii)} $v, v',\cdots,v^{(m-1)}\in  {\rm AC}(\bar \Omega),$  with $v(x):=({ }^{C}\! D_{-1+}^{\mu} u)(x);$ and {\rm (iii)}
% $v^{(m)}\in {\rm BV}(\bar \Omega).$

% where $V_f(\bar\Omega)$ is the total variation of $f$ on $\bar \Omega,$ see Subsection \ref{Subsection2.1}. }
 Note that for the example \eqref{newA}, $v_\mu$ is sufficiently smooth so we have $m=\infty.$ Under this assumption, we can update
the formula  \eqref{HatUnCaseC-3Old-1} as follows.
\begin{thm}\label{Thm:52} Assume that $u$ is  of  {\rm AC-BV$_{\!\!\mu,m}$-regularity}.
Then  for $n\ge \mu+m+1,$ we have
		\begin{equation}\label{Unend-0}
		\begin{split}
		\hat u_n^{L}& = \frac{2n+1}{2} \bigg\{  \sum_{l=1}^{m}(I_{1-}^{\mu+l+1} P_n)(-1)v^{(l)}_\mu(-1+)+\int_{-1}^{1}(I_{1-}^{\mu+m+1} P_n)(x) \, {\rm d}\{ v^{(m)}_\mu(x)\}\bigg\},
		\end{split}
		\end{equation}
where $(I_{1-}^{\mu+l+1} P_n)(-1)$ has the explicit value given by \eqref{HatUnCaseC-3Old-1+1}.
%\begin{equation}\label{newconstants-1}
%\begin{split}
% C_{n,\beta}:=  - 2^{\beta}\Gamma(\beta+1) \frac{\sin(\beta\pi)} \pi \frac{(2n+1)\Gamma(n-\beta)}{\Gamma(n+\beta+2)}.
%\end{split}\end{equation}
	\end{thm}
	\begin{proof} Since  $v_\mu, \cdots,v^{(m-1)}_\mu\in {\rm AC}(\bar \Omega),$ we can conduct  integration by parts upon \eqref{HatUnCaseC-3Old-1}:
	\begin{equation*}\label{HatUnCaseC-3Old-10}
\begin{split}
\frac{2\,  \hat u_n^L} {2n+1} & =
 (I_{1-}^{\mu+1} P_n)(-1)\, v_\mu(-1)
+\int_{-1}^{1} (I_{1-}^{\mu+1} P_n)(x) \, v_\mu'(x)\,  \rd x\\
& =\cdots=\sum_{l=1}^{m-1}(I_{1-}^{\mu+l+1} P_n)(-1)v^{(l)}_\mu(-1)+\int_{-1}^{1}(I_{1-}^{\mu+m} P_n)(x) v^{(m)}_\mu(x)  \,  \rd x\\
& = \sum_{l=1}^{m}(I_{1-}^{\mu+l+1} P_n)(-1)v^{(l)}_\mu(-1+)+\int_{-1}^{1}(I_{1-}^{\mu+m+1} P_n)(x) \, {\rm d}\{ v^{(m)}_\mu(x)\},
\end{split}
\end{equation*}
where the boundary values at $x=1$ vanish in view of \eqref{twoGGFs}, and in the last step, we used the factor
$v^{(m)}_\mu\in {\rm BV}(\bar \Omega)$ and \eqref{IPPW1122}.
			\end{proof}

Comparing the formulas of $\hat u_n^L$ in Theorem \ref{IdentityForUn} (with $\theta\to -1^+$, i.e., \eqref{HatUnCaseC-3Old-1})  and Theorem \ref{Thm:52},
we find they  largely differ from the regularity index.  We can use Lemmas \ref{BoundIntLeg} and \ref{Pnfyi} to deal with the fractional integrals of the Legendre polynomial. Accordling,  we can follow the same lines as in the proofs of Theorems \ref{TruncLeg} and \ref{newL2est} to derive the following estimates. To avoid the repetition, we skill the proof, though there is  subtlety in some derivations.
\begin{thm}\label{TruncLegEnd} Assume that $u$ is  of   {\rm AC-BV$_{\!\!\mu,m}$-regularity}. Then  we have the following estimates.
%
%	Let  $\mu=m+s, $ $\zeta=k+\mu, $
%	$s,\mu\in (0,1),$  $m,k\in {\mathbb N}_0.$
%		If $u\in {\mathbb  W}^{\mu}_{-}(\Omega)$, $v:={}_{-1}^R D_{x}^{s}u^{(m)}(x)\in {\mathbb  W}^{\zeta}_{-}(\Omega),$ we have the following estimates.
	\begin{itemize}
\item[(i)] For $\mu>1/2$ and $N\ge \mu+m,$
		\begin{equation}\label{FracEndLinfA}
		\begin{split}
		\|u-\pi_N^L  u \|_{L^\infty (\Omega)}&\le \bigg\{\frac{1}{2^{\mu+m-1} (\mu+m-1/2)\sqrt{\pi}}\frac{ \Gamma( ({N-\mu-m+1})/ 2)}{\Gamma( ({N+\mu+m})/2)}\\
&\qquad  +\sum_{j=0}^{m}\frac{2^{\mu+j}\Gamma(\mu+j+1)}{\pi(\mu+j-1)}\frac{\Gamma(N-\mu-j+1)}{\Gamma(N+\mu+j+1)}
\bigg\}\, U^{(\mu,m)}_-.
		\end{split}
		\end{equation}
	%\item[(iii)] For $1 <\mu\le \mu+m\le N+1,$
%		\begin{equation}\label{FracEndLinfB}
%		\begin{split}
%		\big\|(1-x^2)^{\frac 14}(u-\pi_N^L  u)\big\|_{L^\infty(\Omega)}\le &U_-^{m,s,m}\bigg\{\frac {1 }{2^{\mu+m-2} (\mu+m-1)\pi}\frac{ \Gamma( ({N-\mu-m})/2+1)}{\Gamma( ({N+\mu+m})/2)}\\
%&\qquad\quad  +\sum_{j=0}^{m}\frac{2^{\mu+j-1/2}\Gamma(\mu+j)}{\sqrt{\pi}\pi(\mu+j-3/2)}\frac{\Gamma(N-\mu-j+2)}{\Gamma(N+\mu+j+1/2)}
%\bigg\}.
%		\end{split}
%		\end{equation}
		\item[(ii)] For $ \mu>-1/2$ and $N> \mu+m,$
		\begin{equation}\label{FracEndL2B}
		\begin{split}
		\|u-\pi_N^L  u\|_{L^2(\Omega)}&\le\bigg\{ \frac{4}{(2\mu+2m+1) \pi}
		\frac{ \Gamma( {N-\mu-m})}{\Gamma( {N+\mu+m+1})}\\
&\qquad + \frac{2^{6\mu+8}\Gamma^2(\mu+1)}{\pi^2(4\mu+2)} \frac{(N+1)^2\,\Gamma(2N-2\mu+1)}{(2N+1)^2\,\Gamma(2N+2\mu+3)}
\bigg\}^{\frac 12}\, U^{(\mu,m)}_-.
		\end{split}
		\end{equation}
	\end{itemize}
	%where  $C$ is a  positive constant independent of $N$ and $u$.
\end{thm}

In contrast to  the interior singularity with a half-order loss, the $L^\infty$-estimate in this case  is optimal. In fact, for the endpoint singularity, the largest error occurs near the boundary where $|P_n(x)|$ attend its maximum  at $x=\pm 1$ (see
Figure \ref{PnWithWeight} (left)), so  the direct summation in e.g., \eqref{FracLinfA-2} will not overestimate.
As an illustration, we consider $u(x)=(1+x)^\mu$ with $\mu\in (k-1,k), k\in \mathbb N$.  From  Theorem \ref{TruncLegEnd} and  \eqref{NRatioGammaV2-1},
we find $\|u-\pi_N^L u\|_{L^\infty (\Omega)}\le C N^{-2\mu}$ and $\|u-\pi_N^L u\|_{L^2 (\Omega)}\le C N^{-2\mu-1}.$
	We tabulate in Table \ref{LinfSetaV2} the errors and convergence order of  Legendre approximations to  $u(x)=(x+1)^\mu$ with various $\mu$, which indicate  the optimal convergence order as predicted.
%\begin{table}[!htbp]
%	\centering
%	\caption{Convergence  of Legendre expansion of  $u=(1+x)^\mu$.} %of  $\|u-\pi_N^C u\|_{L^\infty (\Omega)}$ and $\|u-\pi_N^C u\|_{L^2_\omega(\Omega)}$.}
%	\footnotesize
%	\begin{tabular}{|c|c|c|c|c|c|c|c|c|c|c|c|c|}
%		\hline
%		\multicolumn{1}{ |c  }{\multirow{2}{*}{$N$} } &
%		\multicolumn{6}{ |c| }{Errors in $L^\infty$-norm} & \multicolumn{6}{ |c| }{Errors in $L^2$-norm}      \\ \cline{2-13}
%		\multicolumn{1}{ |c  }{}                        &
%		\multicolumn{1}{ |c| }{$\mu=0.1$}  & order
%		&$ \mu=1.2$  & order  & $\mu=2.6$  & order  & $\mu=0.1$  & order
%		&$ \mu=1.2$  & order & $\mu=2.6$  & order \\ \hline
%$ 2^3 $ &6.15e-01 &--& 2.27e-3 &--& 1.20e-4 &-- &8.82e-3 &-- & 2.32e-04 &-- & 1.98e-05 &-- \\
%$ 2^4 $ &5.41e-01 & 0.18 & 4.87e-04 & 2.22 & 4.03e-06 & 4.90 &4.11e-03 & 1.10 & 2.64e-05 & 3.14 & 3.50e-07 & 5.82 \\
%$ 2^5 $ &4.74e-01 & 0.19 & 9.87e-05 & 2.30 & 1.25e-07 & 5.01 &1.85e-03 & 1.15 & 2.75e-06 & 3.26 & 5.59e-09 & 5.97  \\
%$ 2^6 $ &4.14e-01 & 0.20 & 1.94e-05 & 2.35 & 3.66e-09 & 5.10 &8.22e-04 & 1.17 & 2.74e-07 & 3.33 & 8.31e-11 & 6.07\\
%$ 2^7 $ &3.61e-01 & 0.20 & 3.74e-06 & 2.37 & 1.03e-10 & 5.14 &3.61e-04 & 1.19 & 2.67e-08 & 3.36 & 1.18e-12 & 6.13 \\
%$ 2^8 $ &3.15e-01 & 0.20 & 7.15e-07 & 2.39 & 2.87e-12 & 5.17 &1.58e-04 & 1.19 & 2.56e-09 & 3.38 & 1.65e-14 & 6.17\\
%		\hline
%	\end{tabular} \label{LinfSetaV2}
%\end{table}	
\begin{table}[!htbp]
	\centering
	\caption{Convergence  of Legendre expansion of  $(1+x)^\mu$.} %of  $\|u-\pi_N^C u\|_{L^\infty (\Omega)}$ and $\|u-\pi_N^C u\|_{L^2_\omega(\Omega)}$.}
	\small
	\begin{tabular}{|c|c|c|c|c|c|c|c|c|}
		\hline
		\multicolumn{1}{ |c  }{\multirow{2}{*}{$N$} } &
		\multicolumn{4}{ |c| }{Errors in $L^\infty$-norm} & \multicolumn{4}{ |c| }{Errors in $L^2$-norm}      \\ \cline{2-9}
		\multicolumn{1}{ |c  }{}                        &
		\multicolumn{1}{ |c| }{$\mu=0.1$}  & order
		&$ \mu=1.2$  & order  &  $\mu=0.1$  & order
		&$ \mu=1.2$  & order \\ \hline
$ 2^3 $ &6.15e-01 &--& 2.27e-3 &--&8.82e-3 &-- & 2.32e-04 &--  \\
$ 2^4 $ &5.41e-01 & 0.18 & 4.87e-04 & 2.22  &4.11e-03 & 1.10 & 2.64e-05 & 3.14 \\
$ 2^5 $ &4.74e-01 & 0.19 & 9.87e-05 & 2.30  &1.85e-03 & 1.15 & 2.75e-06 & 3.26 \\
$ 2^6 $ &4.14e-01 & 0.20 & 1.94e-05 & 2.35  &8.22e-04 & 1.17 & 2.74e-07 & 3.33 \\
$ 2^7 $ &3.61e-01 & 0.20 & 3.74e-06 & 2.37  &3.61e-04 & 1.19 & 2.67e-08 & 3.36 \\
$ 2^8 $ &3.15e-01 & 0.20 & 7.15e-07 & 2.39  &1.58e-04 & 1.19 & 2.56e-09 & 3.38 \\
		\hline
	\end{tabular} \label{LinfSetaV2}
\end{table}

\subsection{Concluding remarks} We presented a new fractional Taylor formula for singular functions whose integer-order  derivatives
up to $k-1$ are absolutely continuous and Caputo fractional derivative of order $\mu\in (k-1,k]$  is of bounded variation. It could be viewed as an ``interpolation" between the  usual Taylor formulas of two consecutive integer orders.   We derived from this remarkable tool a similar fractional representation of the Legendre expansion of this type of functions, which became the cornerstone of the optimal  error estimates for  the Legendre orthogonal projection. The set of results under the fractional AC-BV framework greatly enriched the approximation theory for spectral and $hp$ methods. It set a good example to show how the fractional calculus could impact this classic field, and seamlessly bridge  between the results  valid only for integer cases.  Here we merely discussed the approximation results, but this will pave the way for the analysis of and applications to  singular problems, which  will be a topic  worthy of future deep investigation.

\begin{appendix}

\section{Useful properties of Gamma function}\label{Gmprop}
\renewcommand{\theequation}{A.\arabic{equation}}
\renewcommand{\thelmm}{A.\arabic{lmm}}
\setcounter{equation}{0}
\setcounter{lmm}{0}
%\setcounter{thm}{0}
%
%In much of our  analysis, Gamma functions are involved in the  bounds or identities,  so we feel compelled to review its properties including particularly some bounds valid   for a small argument. Indeed, we aim to obtain error bounds valid  for even small discretization parameters.

 % {\color{red} \bf We shall also only keep the necessary ones!}
% will involve the Gamma function  defined by
%\begin{equation*}
%\Gamma(x)=\int_0^\infty t^{x-1} e^{-t} {\rm d} t,\quad x>0,
%\end{equation*}

%Recall the properties (\cite[Ch.2]{Andrews1999Book})
%\begin{equation}\label{Nist15421cc}
%\lim_{z\to 1^-}\frac{ {}_2F_1(a,b;c;z)}{(1-z)^{c-a-b}}=\frac{\Gamma(c)\Gamma(a+b-c)}{\Gamma(a)\Gamma(b)},\quad {\rm if}\;\; c<a+b;
%\end{equation}
Recall the Euler's reflection formula   (cf.\! \cite[Ch.2]{Andrews1999Book}):
\begin{equation}\label{nonitA}
\Gamma(1-a)\Gamma(a)=\frac{\pi} {\sin (\pi a)},\quad a\not=\pm 1,\pm 2,\cdots,
\end{equation}
and the Legendre duplication formula (cf.\! \cite[(5.5.5)]{Olver2010Book}):
	\begin{equation}\begin{split}\label{DuplicationFormula}
	\Gamma(2z)=\pi^{-1/2}2^{2z-1}\Gamma(z)\Gamma(z+1/2).
	\end{split}\end{equation}
%According to  \cite[(6.1.38)]{Abramowitz1972Book}, we have
% 	\begin{equation}\label{RateGamma-0-4}
%	\Gamma(z+1)=\sqrt{2 \pi} z^{z+1 / 2} \exp \Big(\!\!-\!z+\frac{\theta}{12 z}\Big),\quad z> 0, \;\; \theta\in (0,1).
%	\end{equation}
%The following  bounds can be founded in \cite[(1.5)]{Batir08AM-Inequalities}:
%	\begin{equation}\label{RateGamma-0-5}
%	\sqrt{2\pi}\, z^{z+1/2} e^{-z}  \sqrt{1+\frac 1{6z}}<\Gamma(z+1) <\sqrt{2\pi}\, z^{z+1/2} e^{-z} \sqrt{1+\frac {e^2} {2\pi z}-\frac 1 z},\quad z\ge 1.
%	\end{equation}
%
%It is noteworthy that  the bounds  \eqref{NRatioGamma-0}  in \cite{Kershaw83MC-Some} are fairly sharp, so are \eqref{NRatioGammaV2-1}.  In Figure \ref{FigForCompare}, we depict two ratios: Ratio$_1=$ upper bound/true value  and Ratio$_2=$ lower bound/true value with different $a,b.$
%
%\begin{figure}[!ht]
%	\begin{center}
%		{~}\hspace*{-20pt}	 \includegraphics[width=0.38\textwidth]{BoundCase1}\qquad\quad
%		\includegraphics[width=0.38\textwidth]{BoundCase2}
%		\caption{Ratios of the bounds for \eqref{NRatioGammaV2-1}.      Left: $a=1,b=1.2.$ Right: $a=1.2,b=5.8$.}
%		\label{FigForCompare}
%	\end{center}
%\end{figure}
From
	\cite[(1.1) and Thm. 10]{Alzer1997MC}, we have  that for $0\le a\le b$, the ratio
	\begin{equation}\label{gammratio}
	{\mathcal R}_b^a(z):=\frac{\Gamma(z + a)}{\Gamma(z + b)},\quad z\ge 0,
	\end{equation}
	is decreasing with respect to $z.$
On the other hand,   the ratio
	\begin{equation}\label{IneqtyLeg-1}
\widehat {\mathcal R}_c(z):=\frac{1}{\sqrt {z+c}}\frac{\Gamma(z+1)}{\Gamma(z+1/2)},
\end{equation}
	is increasing (resp. decreasing) on $[-1/2,\infty)$ (resp.  $(-c,\infty)$),  if $c\ge 1/2$ (resp. $c\le 1/4$),
based on   \cite[Corollary 2]{Bustoz1986MC}.

In the error bounds, the ratio of two Gamma functions appears very often, so the following inequality is useful.
\begin{lemma}\label{Gammalm} Let $b\in (a+m, a+m+1)$ for some  integer $m\ge 0,$ and
set  $b=a+m+\mu$ with $\mu\in (0,1).$
 Then for $z+a>0$ and $z+b>1, $ we have
\begin{equation}\label{NRatioGammaV2-1}
\frac 1 {(z+a)_{m}}\Big(z+b-\frac{3}{2}+\Big(\frac 5 4-\mu\Big)^{1 / 2}\Big)^{-\mu}< \frac{\Gamma(z+a)}{\Gamma(z+b)}< \frac 1 {(z+a)_{m}} \Big(z+b-\frac{\mu+1}{2}\Big)^{-\mu},
\end{equation}
where the Pochhammer symbol: $(c)_{m}=c(c+1)\cdots (c+m-1).$
%is defined by
%%For $a\in {\mathbb R},$  the rising factorial in the Pochhammer symbol is defined by {\color{red} \bf We shall put the formulas to be used in this part in the proof!}
%\begin{equation}\label{anotation}
%(a)_0=1; \;\;\;\;  % (a)_j=a(a+1)\cdots (a+j-1),\;\;  {\rm for}\;\; j>0;\quad
%(a)_j=a(a+1)\cdots (a+j-1), \;\;\;\forall\, j\in {\mathbb N}. %=\frac{\Gamma(a+j)}{\Gamma(a)},\;\; {\rm for}\;\; j\ge 1.
%\end{equation}
\end{lemma}	
\begin{proof} In fact,  \eqref{NRatioGammaV2-1} can be derived from the bounds in  \cite[(1.3)]{Kershaw83MC-Some}:
%\begin{equation}\label{RatioGamma-0}
%\big(x+\frac{s}{2}\big)^{1-s}<\frac{\Gamma(x+1)}{\Gamma(x+s)}<\Big(x-\frac{1}{2}+\Big(s+\frac{1}{4}\Big)^{1 / 2}\Big)^{1-s}.
%\end{equation}
\begin{equation}\label{NRatioGamma-0}
\Big(x-\frac{1}{2}+\Big(\nu+\frac{1}{4}\Big)^{1 / 2}\Big)^{\nu-1}<\frac{\Gamma(x+\nu)}{\Gamma(x+1)}<\Big(x+\frac{\nu}{2}\Big)^{\nu-1},
\quad x> 0, \;\;  \nu \in (0,1).
\end{equation}
Indeed,  using the property  $\Gamma(z+1)=z\Gamma(z),$ we can write
\begin{equation*}
\frac{\Gamma(z+a)}{\Gamma(z+b)}= \frac 1 {(z+a)_{m}}\frac{\Gamma(z+a+m)}{\Gamma(z+b)}
= \frac 1 {(z+a)_{m}}\frac{\Gamma(z+b-\mu)}{\Gamma(z+b)}.
\end{equation*}
Then by  \eqref{NRatioGamma-0} with $x=z+b-1$ and $\nu=1-\mu,$  we obtain \eqref{NRatioGammaV2-1} immediately.
\end{proof}

\section{Proof of Lemma  \ref{FracIntPart-0} }\label{AppendixA}
\renewcommand{\theequation}{B.\arabic{equation}}

%We find from  \eqref{obsvers}  and \eqref{IdentityForUn-2},  $g'(x)$ (resp. $h'(x)$)  is  continuous on $(a,b] $ (resp. $[b, 1)$), and  they are also %integrable when $\sigma>0.$
For $f\in L^1(\Omega)$ and $g\in {\rm AC}(\bar \Omega) $, changing the order of integration by the Fubini's Theorem, we derive from  \eqref{leftintRL}  that
%\footnote{\color{red} We might summarize this part as a Lemma on $[a,b]$ for functions in suitable spaces? Replace Lemma \ref{IntByPartsInW11}  by this and the results in this lemma can be cited when we prove the new lemma. }
\begin{equation*}\label{FracIntPart-3}
\begin{split}
& \int^{b}_{a}  f(x)  I_{a+}^{\rho} g'(x)\,{\rm d}x=\frac{1}{\Gamma(\rho)}\int_{a}^b  \bigg\{\int _{a}^x \frac{g'(y)}{(x-y)^{1-\rho}} {\rm d}y\bigg\}  f(x) \, {\rm d}x\\
&=\frac{1}{\Gamma(\rho)}\int_{a}^b  \bigg\{\int _{y}^b \frac{f(x)}{(x-y)^{1-\rho}} {\rm d}x\bigg\}  g'(y) \, {\rm d}y=\frac{1}{\Gamma(\rho)}\int_{a}^b  \bigg\{\int _{x}^b \frac{f(y)}{(y-x)^{1-\rho}} {\rm d}y\bigg\}  g'(x) \, {\rm d}x
\\
& =\int^{b}_{a} g'(x)\, I_{b-}^{\rho} f(x)\,{\rm d}x.
%=&\big\{g(x) \, {}_{x}I_{b}^{1-\mu} f(x)\big\}\big|^b_{a}
%+ \int^{b}_{a}   g(x) \, {}_{x}^R D_{b}^{\mu}  \, f(x) \, {\rm d}x.
\end{split}
\end{equation*}
%Similarly, we can show that
%\begin{equation}\label{fgdxds}
%\int^{1}_{\theta} f(x)\,  I_{1-}^{1-s} h'(x)\,{\rm d}x=  \int_{\theta}^{1} h'(x) \,  I_{\theta+}^{1-s} f(x)\,{\rm d}x.
%\end{equation}
If $ I_{b-}^{\rho} f(x)\in {\rm BV}(\bar\Omega),$ we  derive from \eqref{IPPW1122} that
\begin{equation*}\label{FracIntPart-4}
\begin{split}
\int^{b}_{a}   f(x) & \,I_{a+}^{\rho} g'(x)\,{\rm d}x  =\int^{b}_{a} g'(x)\,  I_{b-}^{\rho} f(x)\,{\rm d}x = \big\{g(x)\,  I_{b-}^{\rho} f(x)\big\}\big|_{a^+}^{b^-} - \int^{b}_{a} g(x) %\, ( I_{b-}^{1-s} f(x))'
\,{{\rm d}\big\{I_{b-}^{\rho} f(x)\big\}}. % {\mathcal D}_{\!b-}^{s} f(x)\,{\rm d}x,
\end{split}
\end{equation*}
This yields \eqref{FracIntPart-1}.

We can derive  \eqref{FracIntPart-2} in a similar fashion.
%\medskip

\section{Proof of Proposition  \ref{boundaryvalue} }\label{AppendixA0}
\renewcommand{\theequation}{C.\arabic{equation}}

 Recall the first mean value theorem for the integral (cf.  \cite[p. 354]{Vladimir2016book}):  Let $f,g$ be Riemman integrable on $[c,d]$, $m=\inf\limits_{x\in [c,d]} f(x),$ and  $M=\sup\limits_{x\in [c,d]} f(x)$. If $g$ is nonnegative (or nonpositive) on $[c,d],$ then
	\begin{equation}\label{fgmeanv}
	\int_c^d f(x)g(x){\rm d}x=\kappa \int_c^d g(x){\rm d}x,\quad \kappa\in [m,M].
	\end{equation}
	Recall that (cf. \cite{Samko1993Book}):   for  $\alpha>-1$ and $\mu\in{\mathbb R}^+,$
	\begin{equation}\label{Lintformu}
	I_{a+}^\mu \, (x-a)^\alpha=  \dfrac{\Gamma(\alpha+1)}{\Gamma(\alpha+\mu+1)} (x-a)^{\alpha+\mu}.
	\end{equation}
	For any $x\in [a,a+\delta], $ we derive from \eqref{leftintRL}, \eqref{fgmeanv} and \eqref{Lintformu} that
	\begin{equation}\label{integbB}
	\begin{split}
	I_{a+}^\mu\, u (x) & =\frac 1 {\Gamma(\mu)}\int_{a}^x \frac{u(y)}{(x-y)^{1-\mu}} {\rm d}y= \frac{\kappa(x)}{\Gamma(\mu)} \int_{a}^x \frac{(y-a)^\alpha}{(x-y)^{1-\mu}} {\rm d}y\\
	&= \kappa(x){}_a I_{x}^\mu \, (x-a)^\alpha =\frac{  \Gamma(\alpha+1)\kappa(x)(x-a)^{\mu+\alpha}}{\Gamma(\alpha+\mu+1)} ,
	\end{split}\end{equation}
	where $\kappa(x)\in [m(x),M(x)]$, $m(x)=\inf\limits_{y\in [a,x]} v(y)$, $M(x)=\sup\limits_{y\in [a,x]} v(x)$. We know that
	\begin{equation}\label{integbC}
	\lim\limits_{x\rightarrow a^+}m(x)=v(a),\quad \lim\limits_{x\rightarrow a^+}M(x)=v(a)\Rightarrow \lim\limits_{x\rightarrow a^+}\kappa(x)=v(a).
	\end{equation}
	From \eqref{integbB} and \eqref{integbC}, we obtain \eqref{integbA}.
	%   \begin{equation}\label{IntegbD}
	% \lim\limits_{x\rightarrow a^+}{}_{a}I_{x}^\mu\, u (x)=
	%\begin{cases}
	%v(a)\Gamma(\alpha+1),&\alpha+\mu=0,\\[8pt]
	%0,&\alpha+\mu>0.
	%\end{cases}
	%\end{equation}
	%Similarly, we get \eqref{integbA-2},
	This completes the proof.

%\medskip

\section{Proof of Theorem  \ref{01estimate} }\label{AppendixC}
\renewcommand{\theequation}{D.\arabic{equation}}
We start with the exact formula for the Legendre expansion coefficients  of $u(x)=|x|:$
\begin{equation}\label{un|x|}
\begin{split}
 \hat u_{2j}^L= \frac{(-1)^{j+1} (j+1/4) \Gamma(j-1/2)}{\sqrt{\pi}\, (j+1)!},\quad \hat u_{2j+1}^L=0,\quad j\ge 1,
\end{split}
\end{equation}
which can be derived from \eqref{HatUnCases1} with $k=1,$ i.e., %{Pnmumu3}
\begin{equation*}
\begin{split}
\hat u_n^L&=\frac{2n+1}{2} \int^{1}_{-1}(I_{1-}^{2} P_n)(x)\, \rd \big\{ u^{(1)}(x)\big\}=\frac{2n+1}{2^{2}} (I_{1-}^{2} P_n)(0)\\
&
=\frac{2n+1}{2^{5}}
{}_2F_1\Big(\!\! -n+2, n+3;3;\frac{1} 2\Big),\quad n\ge 2,
\end{split}
\end{equation*}
and the value  at $z=1/2$ (cf. \cite[(15.4.28)]{Olver2010Book}): %(cf.  \cite[(P. 148)]{Andrews1999Book})
\begin{align*}\label{Fatzero}
{}_2F_1\Big(a,b;\frac{a+b+1}2;\frac 1 2\Big) =
\frac{\sqrt{\pi}\,\Gamma((a+b+1)/2)}{\Gamma((a+1)/2)\Gamma((b+1)/2)}.
\end{align*}
Then we obtain from \eqref{un|x|} that
\begin{equation}\label{un|x|-4}
\begin{split}
(u -\pi_N^L  u)(0)&=\sum_{j=\lceil \frac {N+1}2 \rceil}^\infty \hat u_{2j}^{L}P_{2j}(0)=- \frac{1 }{\pi}
\sum_{j=\lceil \frac{N+1} 2 \rceil}^\infty  \frac{(j+1/4)\Gamma(j-1/2)\Gamma(j+1/2)} {\Gamma(j+1)\Gamma(j+2)},
\end{split}
\end{equation}
where $\lceil \frac{N+1} 2 \rceil$ is the  smallest integer   $\ge \frac{N+1} 2,$ and we used the known value (cf.
\cite{Szego1975Book}):
%by (2.40) in \cite{Liu19MC-Optimal}, we have
\begin{equation*}\label{P2kvalue}
P_{2j}(0)={}_2F_1\Big(\!\! -2j, 2j+1;1;\frac{1} 2\Big)=(-1)^j\frac{\Gamma( j+1/2)}{\sqrt{\pi}\,j!}.
\end{equation*}
From \eqref{gammratio}, we have
\begin{equation}\label{un|x|-5}
\frac{\Gamma(j+1/2)}{\Gamma(j+1)}\le \frac{\Gamma(j)}{\Gamma(j+1/2)}.
\end{equation}
Thus, using \eqref{un|x|-5} and $\Gamma(z+1)=z\Gamma(z),$ we obtain
\begin{equation*}\label{un|x|-6}
\begin{split}
\frac{(j+1/4)\Gamma(j-1/2)\Gamma(j+1/2)} {\Gamma(j+1)\Gamma(j+2)}&=\frac{j+1/4}{j+1} \frac{\Gamma(j-1/2)}{\Gamma(j+1)}\frac{\Gamma(j+1/2)} {\Gamma(j+1)} \le \frac{\Gamma(j-1/2)\Gamma(j)} {\Gamma(j+1/2)\Gamma(j+1)}\\
&=\frac{1}{(j-1/2)j}\le \frac{1}{(j-1)j}=\frac{1}{j-1}-\frac{1}{j} .
\end{split}
\end{equation*}
Then
\begin{equation*}\label{un|x|-7}
\begin{split}
&\sum_{j=\lceil \frac{N+1} 2 \rceil}^\infty  \frac{(j+1/4)\Gamma(j-1/2)\Gamma(j+1/2)} {\Gamma(j+1)\Gamma(j+2)}
\le \sum_{j=\lceil \frac {N+1}2 \rceil}^\infty \Big(\frac{1}{j-1}-\frac{1}{j}\Big) =\frac{1}{\lceil \frac{N+1} 2 \rceil-1}\le \frac{2}{N-1}.
\end{split}
\end{equation*}
From \eqref{un|x|-4} and the above, we get the first result in \eqref{un|x|-1+1}.
%\begin{rem}\label{getitback}
%\medskip

We  now prove the second estimate in \eqref{un|x|-1+1}.
As $P_n(\pm1)=(\pm1)^n,$ we derive from \eqref{un|x|} that
%$$\big|(u -\pi_N^L  u)(\pm 1)\big| =\big|\sum_{n=}\frac{(-1)^{j+1} (j+1/4) \Gamma(j-1/2)}{\sqrt{\pi}\, (j+1)!} \big| \le  (j+1/4) \Gamma(j-1/2)}{\sqrt{\pi}\, (j+1)!}$$
\begin{equation}\label{un|x|-0}
\begin{split}
(u -\pi_N^L  u)(\pm 1)&=\sum_{j=\lceil \frac {N+1} 2 \rceil}^\infty \hat u_{2j}^{L}\,P_{2j}(\pm1)=
\sum_{j=\lceil \frac {N+1} 2 \rceil}^\infty  \frac{(-1)^{j+1}}{\sqrt{\pi}}\frac{(j+1/4)\Gamma(j-1/2)} {\Gamma(j+2)}.
\end{split}
\end{equation}
Denoting
\begin{equation*}
\begin{split}
{\mathcal S}_j:=  \frac{(-1)^{j+1}}{\sqrt{\pi}}\frac{(j+1/4)\Gamma(j-1/2)} {\Gamma(j+2)},\quad {\mathcal T}_j:=
\frac{1 }{2\sqrt{\pi}}\frac{\Gamma(j-3/2)} {\Gamma(j)},
\end{split}
\end{equation*}
we have
\begin{equation}\label{un|x|V2-1}
\begin{split}
{\mathcal S}_j+{\mathcal S}_{j+1}&=(-1)^{j+1}\frac{3 }{2\sqrt{\pi}}\frac{(j+3/4)\Gamma(j-1/2)} {\Gamma(j+3)} \le \frac{3 }{2\sqrt{\pi}}\frac{\Gamma(j-1/2)} {\Gamma(j+2)}\\
&\le  \frac{3 }{4\sqrt{\pi}}\Big(\frac{\Gamma(j-3/2)} {\Gamma(j+1)}+\frac{\Gamma(j-1/2)} {\Gamma(j+2)}\Big)= \big({\mathcal T}_j-{\mathcal T}_{j+1}\big)+
\big({\mathcal T}_{j+1}-{\mathcal T}_{j+2}\big),
\end{split}
\end{equation}
where we noted
\begin{equation*}
\begin{split}
 \frac{\Gamma(j-1/2)} {\Gamma(j+2)}\le  \frac{\Gamma(j-3/2)} {\Gamma(j+1)},
\end{split}
\end{equation*}
and
\begin{equation*}
\begin{split}
 \frac{3 }{4\sqrt{\pi}}\frac{\Gamma(j-3/2)} {\Gamma(j+1)}= \frac{1 }{2\sqrt{\pi}}\bigg(j \frac{\Gamma(j-3/2)} {\Gamma(j+1)}-(j-3/2)\frac{\Gamma(j-3/2)} {\Gamma(j+1)}\bigg)={\mathcal T}_j-{\mathcal T}_{j+1}.
\end{split}
\end{equation*}
Thus from \eqref{gammratio} and \eqref{un|x|-0}-\eqref{un|x|V2-1}, we obtain
\begin{equation*}\label{un|x|-2}
\begin{split}
\big|(u -\pi_N^L  u)(\pm 1)\big|&= \big(|{\mathcal S}_{\lceil\frac{N+1} 2 \rceil}+{\mathcal S}_{\lceil\frac{N+1} 2 \rceil+1}|\big)+\cdots
+\big(|{\mathcal S}_{\lceil\frac{N+1} 2 \rceil+2i}+{\mathcal S}_{\lceil\frac{N+1} 2 \rceil+2i+1}|\big)+\cdots\\
 &\le \big\{\big({\mathcal T}_{{\lceil\frac{N+1} 2 \rceil}}-{\mathcal T}_{{\lceil\frac{N+1} 2 \rceil}+1}\big)+
\big({\mathcal T}_{{\lceil\frac{N+1} 2 \rceil}+1}-{\mathcal T}_{{\lceil\frac{N+1} 2 \rceil}+2}\big)\big\}+\cdots\\
&\quad + \big\{\big({\mathcal T}_{{\lceil\frac{N+1} 2\rceil}+2i }-{\mathcal T}_{{\lceil\frac{N+1} 2 \rceil}+2i+1}\big)+
\big({\mathcal T}_{{\lceil\frac{N+1} 2 \rceil}+2i+1}-{\mathcal T}_{{\lceil\frac{N+1} 2 \rceil}+2i+2}\big)\big\}+\cdots\\
 &= \sum_{j=\lceil \frac{N+1} 2 \rceil}^\infty\big({\mathcal T}_j-{\mathcal T}_{j+1}\big)=\frac{1 }{2\sqrt{\pi}}\frac{\Gamma(\lceil \frac {N+1}2 \rceil-3/2)} {\Gamma(\lceil \frac{N+1}2 \rceil)}\le
\frac{1 }{2\sqrt{\pi}}\frac{\Gamma( N/2-1)} {\Gamma( N/2+1/2)}.
\end{split}
\end{equation*}
This ends the proof.

\end{appendix}

 \end{document}